\documentclass{article}

\usepackage{amsmath, amsfonts, amsthm, amssymb, amsrefs,physics,mathtools,xcolor} 
\usepackage{stmaryrd} 
\usepackage[pagewise]{lineno}

\newtheorem{theorem}{Theorem}[section]

\newtheorem{lemma}[theorem]{Lemma}
\newtheorem*{remark}{Remark}

\newcommand{\boldu}{\mathbf{u}}
\newcommand{\Reals}{\mathbb{R}}
\newcommand{\Rn}{\mathbb{R}^n}

\title{Instability of boundary equilibrium for Reaction-Diffusion system of a complex-balanced reaction network}
\author{Jiaxin Jin}

\begin{document}

\maketitle

\begin{abstract}
\noindent In this paper we study the local instability to the boundary equilibria and the local stability to the positive equilibria for some chemical reaction-diffusion systems.
We first analyze instablity of three-species system with boundary equilibria in some stoichiometric classes. Then we prove the convergence to the positive equilibria for a general reversible reaction-diffusion network as long as the initial data is closed enough to the the positive equilibria.

\end{abstract}

\section{Introduction} 
The dynamical behaviour of the mass-action reaction systems has been studied over the last fifty years. The work of Horn, Jackson and Feinberg \cite{Horn.1972, Horn.1974, Feinberg.1972} has been successful in showing existence and stability of positive equilibria. Their work shows that the complex balanced mass-action systems have unique positive equilibria which is locally asymptotically stable independently of reaction rate constant values. Moreover, Horn conjectured that the unique equilibria are in fact globally asymptotically stable \cite{Horn.1974} which is known as the {\em Global Attractor Conjecture}. The latest result is a proposed proof of the Global Attractor Conjecture in full generality \cite{Craciun.gac}.   

For the corresponding reaction-diffusion models, many recent papers have focused on extending the results above in the PDE setting.
A promising way to connect the PDE with ODE models is by using entropy techniques. Recent results by Desvillettes, Fellner and collaborators \cite{DFPV07, DFT2} showed that in the absence of boundary equilibria or special cases of networks with boundary equilibria, the positive equilibrium of the complex balanced reaction-diffusion system attracts all solutions with positive initial data.
The recent paper by Pierre et al \cite{PSU17} studies the general case of a reversible reaction; the authors prove that if the solution is globally (in time) essentially bounded, the solution converges exponentially to the complex-balanced equilibrium. 
However, the general case of systems with boundary equilibria remains open, and the analysis of such systems is on a case-by-case basis.
The most general result on the convergence to equilibrium is \cite{FT18} show that if the system does not have boundary equilibria, then any renormalised solution converges exponentially to the complex balanced equilibrium with a rate, which can be computed explicitly by applying the so-called entropy method, the author.

The network $A+2B\rightleftharpoons B+C$ was considered in \cite{GJCA2019} where it was shown that in one spatial dimension space solutions converge asymptotically to the unique positive equilibrium at explicit rates. However, in higher dimension without globally (in time) essentially boundness the global behaviour of solutions is unknown. Therefore we are interested in local behaviour (close-to-equilibrium regularity) instead. 
The paper \cite{Bao2018} shows that in small dimensions, strong solutions exist for systems with restricted power of non-linear polynomial provided that the initial data is close enough to the equilibrium in $L^{2}$ sense.
In another recent paper \cite{CMT2019}, authors prove that
as long as the closeness to equilibrium is measured in $L^{\infty}$ norm, the convergence holds for arbitrary dimension.

Our paper studies the network $A+2B\rightleftharpoons B+C$ but in three dimensional space by using the elliptic and energy estimate to show that the unique boundary equilibria is locally instable (Theorem \ref{instable theorem}). 
Then for the general case of one reversible pair of reaction 
$
\alpha_1 A_1 + ... +  \alpha_n A_n \rightleftharpoons \beta_1 A_1 + ... +  \beta_n A_n
$,
we use the same technique but be able to show the locally stability of the unique positive equilibria (Theorem \ref{stable theorem in general}).
It worthes mentioning that the local stability around positive equilibrium in $L^{\infty}$ norm for this system can be achieved by \cite{CMT2019} and \cite{PSU17}. This paper provides a different method to show  the local stability in $H^2$ norm. Moreover we use this method using elliptic and energy estimate to show local instability around boundary equilibrium.

In the remainder of this introductory section we set up terminology and notation, we discuss some of the techniques used here and in previous work, and we state our main theorems. Sections 2 contains the proofs of the results for the local instability for $A+2B\rightleftharpoons B+C$ and Section 3 contains the proofs of local stability for $
\alpha_1 A_1 + ... +  \alpha_n A_n \rightleftharpoons \beta_1 A_1 + ... +  \beta_n A_n
$.

\subsection{Terminology and previous results}

We consider $0<T\leq\infty$ and a semilinear parabolic system 
\begin{equation*} 
\boldu_t-\mathcal{D}\Delta \boldu=R(\boldu)\mbox{ in }\Omega\times(0,T)
\end{equation*}
with an initial data $$\boldu(\cdot,0)=\boldu_0\mbox{ in }\Omega,$$
where $\boldu:\Omega\times[0,T)\rightarrow\Rn$ is a vector of concentrations at spatial position $x\in \Omega$  (an open and bounded subset of $\mathbb{R}^3$) and time $t\in[0,\infty)$, $\mathcal{D}$ is a positive definite, diagonal $n\times n$ matrix and we consider Neumann boundary conditions throughout this work:
$$\frac{\partial u_i}{\partial \mathbf{n}} := \nabla u_i\cdot \mathbf{n} =0\mbox{ on }\partial\Omega\times(0,T),\ i=1,...,n,$$
where $\mathbf{n}$ is the outer normal vector at the boundary and $R:\Rn\rightarrow \Rn$ is a vector field whose components are polynomials and it is determined by the chemical reactions under consideration.

For example, the single reaction $A+B\overset{k}{\rightarrow} C$.
Here $A,\ B,$ and $C$ are the three {\em species} of the network, and $A+B$ and $C$ are its {\em complexes}. In general, complexes are formal linear combinations of species with non-negative integer coefficients, and sit on both sides of a reaction arrow. It is useful to think of complexes as vectors in a natural way, for example $A+B$ corresponds to $y=(1,1,0)$, and $C$ to $y'=(0,0,1)$. 
The concentrations of $A,\ B,\ C$ are non-negative functions of time and space and are collected in the {\em concentration vector} $\boldu=(a,b,c)$. 
The {\em reaction rate} of a reaction is given by mass-action, and is proportional to the concentration of each reactant species. This way, the reaction $A+B\overset{k}{\rightarrow} C$ has rate $kab$. The {\em reaction rate constant} $k$ is a reaction-specific positive number. 
In general, the rate of a the reaction $y\overset{k}{\to} y'$ is given by 
$$k\boldu^{y}=k\displaystyle\prod_{i=1}^nu_i^{y_i},$$  
where $n$ is the number of species, and complexes $y$ and $y'$ are viewed as vectors, as discussed above. 
Therefore, this is given by        
$$R(\boldu):=\sum_{y\to y'} k_{y\to y'}\boldu^{y}(y'-y),$$
where $k_{y \to y'}$ is the rate constant of $y\to y'$ and the summation is over all reactions $y\to y'$ in the network.

In general, we say that an equilibrium point $u_0$ of a {\em reaction system} (i.e. an equilibrium of the ODE system ${\bf u}_t=R({\bf u})$ without diffusion) is a {\em complex balanced equilibrium} if for all complexes $\bar y$ we have  
$$\sum_{\bar y \to y} k_{\bar y \to y}u_0^{\bar y} = \sum_{y \to \bar y} k_{y \to \bar y}u_0^{y}.$$ 
A reaction system is called {\em complex balanced} if it admits a positive complex balanced equilibrium. We call a reaction-diffusion system complex balanced if its corresponding reaction system is complex balanced. It was shown that all steady states of a complex balanced reaction-diffusion system are constant functions (do not depend on space), whose values equal the steady states of the corresponding (complex balanced) reaction system \cite{Mincheva}. We can therefore identify the steady states (equilibria) of complex balanced reaction-diffusion systems with those of corresponding reaction systems.

Reaction systems often admit linear first integrals, called {\em conservation laws}; for example, the system $A+2B \rightleftharpoons B+C$ has conservation laws $a+c$= {\em const.} and $b+c$= {\em const}.
In this paper, an {\em accessible boundary equilibrium} of a reaction network is an equilibrium on the boundary of the positive orthant which gives the same values of the conservation laws as some phase point with strictly positive coordinates. In paper \cite{DFT2} it was shown that for complex-balanced reaction-diffusion systems without accessible boundary equilibria, certain existence conditions imply convergence of solutions to positive equilibria. The reaction-diffusion systems we consider in this paper are complex-balanced with accessible boundary equilibria.

For complex balanced systems, recent work by Craciun \cite{Craciun.gac} proved the Global Attractor Conjecture which states that regardless of the existence of boundary equilibria, trajectories starting in the positive orthant converge to the unique positive equilibrium in the corresponding stoichiometric class.  
In the PDE case, the most general result concerns the case where there are no boundary equilibria. Desvillettes, Fellner and Tang \cite{DFT2} showed that under some initial condition, (weak or renormalized) solutions exponentially fast converge to the equilibrium which lies in the same stoichiometric class as the initial data via the use of {\em entropy-entropy dissipation inequality} (EEDI). Very recently, Cupps, Morgan and Tang \cite{CMT2019} showed that if initial condition is closed enough to the positive equilibria in $L^{\infty}$ sense, strong solutions will  exponentially converge to the equilibrium in $L^{\infty}$ sense via the duality method and the regularization of the heat operator.

\subsection{Instability}

The system we consider in this section is $A+2B \rightleftharpoons B+C$. 
The choice of spatial dimension $d=3$ and we assume $\Omega$ is connected and bounded domain in $\mathbb{R}^3$. 
In this paper, Our method  to prove both stability and instability seems confined to three-D or lower, as it uses the Sobolev embedding inequality where the $H^2$ norm of the solutions leads to the boundness of the $L^\infty$ norm.
Notice that by rescaling time $t$, space $x$ and the concentrations $(a,b,c)$, from \cite{GJCA2019} we can always assume that reaction rates and domain volume are 1.
$$
A+2B \rightleftharpoons B+C.
$$
The corresponding $3 \times 3$ reaction-diffusion system is
\begin{equation}
\begin{cases} \label{system}
\Tilde{a}_t-d_a \Delta \Tilde{a}= - \Tilde{a}\Tilde{b}^2+\Tilde{b}\Tilde{c} &  x \in \Omega, t > 0 \\
\Tilde{b}_t-d_b \Delta \Tilde{b}=- \Tilde{a}\Tilde{b}^2+\Tilde{b}\Tilde{c} &  x \in \Omega, t > 0 \\
\Tilde{c}_t-d_c \Delta \Tilde{c}=\Tilde{a}\Tilde{b}^2-\Tilde{b}\Tilde{c}  &  x \in \Omega, t > 0 \\
\frac{\partial \Tilde{a}}{\partial \mathbf{n}}= \frac{\partial \Tilde{b}}{\partial \mathbf{n}}=\frac{\partial \Tilde{c}}{\partial \mathbf{n}} =0 &  x \in \partial \Omega, t > 0\\
\Tilde{a}(x,0)=\Tilde{a}_0(x),\Tilde{b}(x,0)=\Tilde{b}_0(x),\Tilde{c}(x,0)=\Tilde{c}_0(x) &  x \in \Omega, \\
\end{cases}
\end{equation}
where $\boldu = (\Tilde{a},\Tilde{b},\Tilde{c})$ stands for the concentration of $(A,B,C)$. In this case $\mathcal{D}=\mathrm{diag}\{d_a,d_b,d_c\}\in M_{3\times 3}(\Reals)$ denotes the diagonal matrix of diffusion constants. 

Considering the reaction system, we have the following conservation laws;
\begin{equation}
\begin{split}  \label{conservation law}
& \int_{\Omega} \Tilde{a}(t,x) \ dx+\int_{\Omega} \Tilde{c}(t,x) \ dx=\int_{\Omega}\Tilde{a}_0(x) \ dx+\int_{\Omega} \Tilde{c}_0(x) \ dx := M_1,
\\& \int_{\Omega} \Tilde{b}(t,x) \ dx+\int_{\Omega} \Tilde{c}(t,x) \ dx=\int_{\Omega}\Tilde{b}_0(x) \ dx+\int_{\Omega} \Tilde{c}_0(x) \ dx 
:= M_2.
\end{split}
\end{equation}
From \cite{GJCA2019}, in the case $A+2B \rightleftharpoons B+C$ as long as $M_1 > M_2$ there are two types of equilibrium  $(a_{\infty},b_{\infty},c_{\infty})$ and $( a_{\infty} ,0,c_{\infty})$ following    \eqref{conservation law},
and we name 
$(a_{\infty},b_{\infty},c_{\infty})$ as unique positive equilibrium and $( a_{\infty} ,0, c_{\infty})$ as unique accessible boundary equilibrium. 
We exclude the case when $\int_{\Omega} b_0 dx = 0$, since in this situation the system degenerates to the heat equation and the solution converges to $({a}_{\infty} ,0, {c}_{\infty})$ because of $b$ being zero.

To show the instability of boundary equilibria, we define 
$y^{\intercal} = (u^{\intercal}, u^{\intercal}_t)$ with $
u =(\Tilde{a} - a_{\infty}, b,\Tilde{c} - c_{\infty})^{\intercal}
$ 
and get the equation for $y$ such that
$
y_t 
= L y + N(y)
$
where $L$ is the linear operator.
Also we introduce two norms 
$$
\|y\| \coloneqq \|u\|_2+\|u_t\|_2 \ \text{and} \  \interleave y \interleave \coloneqq \|u\|_{H^2}+\|u_t\|_2
$$
where $\| \cdot \|_2$ represents the $L^2$ norm and $\| \cdot \|_{H^2}$ represents the $H^2$ norm.
Our method to prove local instability for $A+2B\rightleftharpoons B+C$, as it first shows that the eigenvalues for operator $L$ is non-positive \cite{Strauss.1992} and then uses 
the energy estimate, elliptic estimate
\cite{ADN.1964},  and the rest is based on the argument of \cite{Guo-Strauss}\cite{GHS}. 
It is an important improvement that we can deal with this quadratic case in higher dimension
, since previous results only dealt with one dimension  \cite{GJCA2019}.

In this paper we will prove the instability statement for accessible boundary equilibria, namely:

\begin{theorem}     \label{instable theorem}
Consider a family of initial data $y^{\delta}(0) = \delta y_0$ with $\| y_0 \| = 1$, $\int_{\Omega} b_0 dx \neq 0$ $(b_0 \geq 0)$ and $\interleave  y_0 \interleave < \infty$ and let $\theta_0$ be a fixed  sufficiently small number. Then if $0 \leq t \leq T^{\delta} \sim \log \frac{\theta_0}{\delta}$, at the escape time
$$
\| y(T^{\delta}) \| \geq \tau_0 > 0
$$
where 
$\tau_0$ depends explicitly on $y_0$ and is independent of $\delta$.
\end{theorem}

\begin{remark}
There exists a constant $C_P$ and $\lambda > 0$ such that 
$
\| e^{Lt} y_0 \| \geq C_P e^{\lambda t}
$
and $T^{\delta} = \frac{1}{\lambda} \log \frac{\theta_0}{\delta}$.
\end{remark}
\begin{remark}
In Section 2, we define $T^{*} = \sup_{t} \{\interleave y \interleave < \sigma \}$ 
where $\sigma$ is bounded and defined in Lemma \ref{sigma estimate}
and we can show that
$T^{*} < T^{\delta}$. Under the Sobolev embedding inequality, this guarantees the global existence of solution up to the escape time.
\end{remark}

We can use the same technique to adapt 
$
A_1 + ... + A_l +2B \rightleftharpoons B + C_1 + ... + C_r
$ which is more generalized
and we define $y^{\intercal} = (u^{\intercal}, u^{\intercal}_t)$ and 
$
u =(\Tilde{a_1} - {a}_{1,\infty},...,\Tilde{a_l} - {a}_{l,\infty}, b, \Tilde{c_1} - {c}_{1,\infty},...,\Tilde{c_r} - {c}_{r,\infty})^{\intercal}
,$
we can prove the similar instability statement for accessible boundary equilibria, namely:

\begin{theorem}     \label{instable theorem in general}
Consider a family of initial data $y^{\delta}(0) = \delta y_0$ with $\| y_0 \| = 1$, $\int_{\Omega} b_0 dx \neq 0$ $(b_0 \geq 0)$ and $\interleave  y_0 \interleave < \infty$ and let $\theta_0$ be a fixed  sufficiently small number. 
Then there exists a constant $C_P$ and $\lambda > 0$ such that
$$
\| e^{Lt} y_0 \| \geq C_P e^{\lambda t}
$$
where L is a linear operator such that 
$
y_t 
= L y + N(y)
$ and if $0 \leq t \leq T^{\delta} = \frac{1}{\lambda} \log \frac{\theta_0}{\delta}$, then at the escape time
$$
\| y(T^{\delta}) \| \geq \tau_0 > 0
$$
where $\tau_0$ depends explicitly on $y_0$ and is independent of $\delta$.
\end{theorem}

The above theorems will be proved in Section 2.

\subsection{Stability}

Initially we consider proving the local stability at positive equilibria for simple case $A+2B\rightleftharpoons B+C$ 
by defining the small perturbation $a=\Tilde{a}-a_{\infty}$, $b=\Tilde{b}-b_{\infty}$, $c=\Tilde{c}-c_{\infty}$ where  $(a_{\infty},b_{\infty},c_{\infty})$ is the unique positive equilibrium with  $a_{\infty}b_{\infty}=c_{\infty}$ and compatible with the conservation law.

Therefore we get the following equation for perturbation
\begin{equation}  \label{abc equation}
\begin{split}
 &   a_t-d_a \Delta a = -(b+b_{\infty})(ab-(c-b_{\infty}a-a_{\infty}b))
\\& b_t-d_b\Delta b  = -(b+b_{\infty})(ab-(c-b_{\infty}a-a_{\infty}b))
\\& c_t-d_c\Delta c  = (b+b_{\infty})(ab-(c-b_{\infty}a-a_{\infty}b))
\end{split}
\end{equation}

By multiplying $b_{\infty}a, a_{\infty}b, c$ on \eqref{abc equation} respectively and integrating over $\Omega$ by parts, we get the first part of the energy estimate.

\begin{equation}
\begin{split}
        & \frac{1}{2}\frac{d}{dt}(b_{\infty}\|a\|_{2}^{2}+a_{\infty}\|b\|_{2}^{2}+\|c\|_{2}^{2}) + (d_ab_{\infty}\|\nabla{a}\|_{2}^{2}+d_ba_{\infty}\|\nabla{b}\|_{2}^{2}+d_c\|\nabla{c}\|_{2}^{2})
        \\& = \int_{\Omega}(b+b_{\infty})(ab-(c-b_{\infty}a-a_{\infty}b))(c-b_{\infty}a-a_{\infty}b)dx
\end{split}
\end{equation}

Next step is the most crucial.
We try to absorb the right hand side by the energy-dissipation term $d_ab_{\infty}\|\nabla{a}\|_{2}^{2}+d_ba_{\infty}\|\nabla{b}\|_{2}^{2}+d_c\|\nabla{c}\|_{2}^{2}$. However we can't apply Poincar\'{e} inequality to compare $\|\nabla{a}\|_{2}$ and $\|{a}\|_{2}$ directly since $\int_{\Omega} a \ dx$, $\int_{\Omega} b \ dx$, $\int_{\Omega} c \ dx$ is unknown. Motivated from the conservation law \eqref{conservation law}, we introduce two new variables   
$
d = a + c, e= b+c
$
where $\int_{\Omega} d \ dx = \int_{\Omega} e \ dx = 0$. Now we can apply Poincar\'{e} inequality on $d$ and $e$ to get
\begin{equation*}
 \|d\|_2 \lesssim \|\nabla{d}\|_2, \ \|e\|_2 \lesssim \|\nabla{e}\|_2
\end{equation*}
In this paper, the notation 
$X \lesssim Y$ means that $X \leq CY$ for some
constant $C > 0$.

Then we analyse the sign status for $d$ and $e$ and use the structure of non linearity
along with Poincar\'{e} inequality, here we let $f = c-b_{\infty}a-a_{\infty}b$ to simplify the notation.

If $f \geq 0, ab \leq 0$ or $f \leq 0, ab \geq 0$, the integrand on the right hand side is non-positive, thus 
$$
\int_{\Omega}(b+b_{\infty})(ab-(c-b_{\infty}a-a_{\infty}b))(c-b_{\infty}a-a_{\infty}b)dx \leq 0
$$

If $f < 0, ab < 0$ which implies $a$ and $b$ have different signs, thus
\begin{equation*}
\begin{split}
        & \int_{\Omega}(b+b_{\infty})(ab-(c-b_{\infty}a-a_{\infty}b))(c-b_{\infty}a-a_{\infty}b)dx
        \\& \lesssim \int_{\Omega}(b+b_{\infty})(ab)^2dx  
        = \int_{\Omega}b \cdot (ab)^2dx + \int_{\Omega}b_{\infty}(ab)^2dx
\end{split}
\end{equation*}

If $b \leq 0$, since $|ab| \lesssim (a-b)^2 = (d-e)^2$,
$$
\int_{\Omega}b \cdot (ab)^2dx + \int_{\Omega}b_{\infty}(ab)^2dx \lesssim \int_{\Omega}(d^4+e^4)dx
$$

If $b > 0$ which implies $a < 0$, we have $b < b-a = e-d \lesssim e^2+d^2+1$,
$$
\int_{\Omega}b \cdot (ab)^2dx + \int_{\Omega}b_{\infty}(ab)^2dx \lesssim \int_{\Omega}(e^2+d^2+1)(d^4+e^4)dx
$$

If $f > 0, ab > 0$ and if $f \geq ab$, the integrand on the right hand side is non-positive, thus 
$$
\int_{\Omega}(b+b_{\infty})(ab-(c-b_{\infty}a-a_{\infty}b))(c-b_{\infty}a-a_{\infty}b)dx \leq 0
$$

If $0 < f < ab$ and if $a>0, b>0$, then we can get $c > b_{\infty}a + a_{\infty}b >0$ which implies $d >a$ and $e > b$ and $ab < |d \cdot e| \lesssim d^2 +e^2$, then we have
\begin{equation*}
\begin{split}
        & \int_{\Omega}(b+b_{\infty})(ab-(c-b_{\infty}a-a_{\infty}b))(c-b_{\infty}a-a_{\infty}b)dx 
        \\& \lesssim \int_{\Omega}(b+b_{\infty})(ab)^2dx  
        = \int_{\Omega}b  (ab)^2dx + \int_{\Omega}b_{\infty}(ab)^2dx \lesssim \int_{\Omega}(e^2+1)(d^4+e^4)dx
\end{split}
\end{equation*}

If $0 < f < ab$ and if $a<0, b<0, c \leq 0$, then we can get $d < a < 0$ and $e < b < 0$ and $ab < |d \cdot e| \lesssim d^2 +e^2$, then we have
\begin{equation*}
\begin{split}
        & \int_{\Omega}(b+b_{\infty})(ab-(c-b_{\infty}a-a_{\infty}b))(c-b_{\infty}a-a_{\infty}b)dx 
        \\& \lesssim \int_{\Omega}(b+b_{\infty})(ab)^2dx  
        \lesssim \int_{\Omega}(d^4+e^4)dx
\end{split}
\end{equation*}

If $0 < f < ab$ and if $a<0, b<0, c >0$, then we have 
$$
ab +(b_{\infty}a + a_{\infty}b) > c > 0
$$
but $ab +(b_{\infty}a + a_{\infty}b) = b \cdot (a+ a_{\infty}) +ab_{\infty} < 0$, this is the impossible case.

After considering all above cases, we can get the following
\begin{equation}
\begin{split}
& \frac{1}{2}\frac{d}{dt}(b_{\infty}\|a\|_{2}^{2}+a_{\infty}\|b\|_{2}^{2}+\|c\|_{2}^{2}) + (d_ab_{\infty}\|\nabla{a}\|_{2}^{2}+d_ba_{\infty}\|\nabla{b}\|_{2}^{2}+d_c\|\nabla{c}\|_{2}^{2})
\\& \lesssim   g(\|a,b,c\|_{\infty}) (\|\nabla{d}\|^2_2 + \|\nabla{e}\|^2_2)
\end{split}
\end{equation}
where $g(\|a,b,c\|_{\infty}) = \|(e^2+d^2)(e^2+d^2+1)\|_{\infty}$ 
and provided that $\|a,b,c\|_{L^{\infty}}$ is small enough such that $g \leq \min(d_ab_{\infty},d_ba_{\infty},d_c)$, we have   
\begin{equation} \label{abc first energy estimate}
\frac{d}{dt}(b_{\infty}\|a\|_{2}^{2}+a_{\infty}\|b\|_{2}^{2}+\|c\|_{2}^{2}) + (d_ab_{\infty}\|\nabla a\|_{2}^{2}+d_ba_{\infty}\|\nabla b\|_{2}^{2}+d_c\|\nabla c\|_{2}^{2})
\leq 0
\end{equation}

Then we apply $\partial_t$ on \eqref{abc equation} and multiply them by $b_{\infty}a_t,a_{\infty}b_t,c_t$ respectively, then integrating over $\Omega$ by parts and sum up all three terms, we get
\begin{equation}
\begin{split}
    & \frac{1}{2}\frac{d}{dt}(b_{\infty}\|a_t\|_{2}^{2}+a_{\infty}\|b_t\|_{2}^{2}+\|c_t\|_{2}^{2}) + (d_ab_{\infty}\|\nabla a_t\|_{2}^{2}+d_ba_{\infty}\|\nabla b_t\|_{2}^{2}+d_c\|\nabla c_t\|_{2}^{2})
    \\& = \int_{\Omega}(b_{\infty}a_t+a_{\infty}b_t-c_t)b_t(c-a_{\infty}b-b_{\infty}a-ab) dx 
 + \int_{\Omega}(b_{\infty}a_t+a_{\infty}b_t-c_t)\Tilde{b}(c_t-a_t\Tilde{b}-b_t\Tilde{a}) dx
\end{split}
\end{equation}

Similarly we analyse the sign status for the following variables
$
d_t = a_t + c_t, e_t= b_t+c_t
$
which is also motivated from the conservation law 
and we also get 
\begin{equation*}
\begin{split}
& \int_{\Omega} d_t \ dx = \int_{\Omega} e_t \ dx = 0,
 \|d_t\|_2 \lesssim \|\nabla{d_t}\|_2, \ \|e_t\|_2 \lesssim \|\nabla{e_t}\|_2
\end{split}
\end{equation*}

Considering all possible cases, we are able to show that
if $\|a,b,c\|_{L^{\infty}}$ is small enough, we have the following 
\begin{equation}  \label{abc second energy estimate}
\frac{d}{dt}(b_{\infty}\|a_t\|_{2}^{2}+a_{\infty}\|b_t\|_{2}^{2}+\|c_t\|_{2}^{2}) + (d_ab_{\infty}\|\nabla a_t\|_{2}^{2}+d_ba_{\infty}\|\nabla b_t\|_{2}^{2}+d_c\|\nabla c_t\|_{2}^{2}) \lesssim 0
\end{equation}
Combing energy estimate \eqref{abc first energy estimate}\eqref{abc second energy estimate} with the elliptic estimate (from Theorem \ref{nirenberg} in Section 2)
$$
\|v_i\|_{H_2} \lesssim \|bc-ab^2-2ab b_{\infty}-b^2a_{\infty}\|_2+\sum\limits_{i=1}^{3}\|\partial_t(v_i)\|_2+\sum\limits_{i=1}^{3}\|v_i\|_2
$$
where $v = (a,b,c)$. Then we have the local stability for $a,b,c$ in $H^2$ sense.

In Section 3, we consider the generalized case for one reversible pair
$$
\alpha_1 A_1 + ... +  \alpha_n A_n \rightleftharpoons \beta_1 A_1 + ... +  \beta_n A_n
$$
The corresponding $n \times n$ reaction-diffusion system is
\begin{equation}
\begin{cases} \label{system n}
\partial_t \Tilde{u}_{i}-d_i \Delta \Tilde{u}_{i} = (\beta_i - \alpha_i)(\Tilde{u}^{\alpha} - \Tilde{u}^{\beta}) &  x \in \Omega, t > 0 \\
\nabla{\Tilde{u}}_{i} \cdot n = 0 &  x \in \partial \Omega, t > 0\\
\Tilde{u}_{i} (x,0)=\Tilde{u}_{i, 0}(x) &  x \in \Omega \\
\end{cases}
\end{equation}
where $\alpha = (\alpha_1,...,\alpha_n)$ and $\beta = (\beta_1,...,\beta_n)$ and $\alpha_i, \beta_j$ are non-negative integers. In this case $\mathcal{D}=\mathrm{diag}\{d_i\}\in M_{n\times n}(\Reals)$ denotes the diagonal matrix of diffusion constants.

From \cite{GJCA2019}, as long as $\int_{\Omega} \Tilde{u}_{i, 0}(x) \ dx > 0$ there must exist the unique positive equilibria
and we name it as
$u_{\infty} = ({u}_{1_{\infty}},{u}_{2_{\infty}},...,{u}_{n_{\infty}})$. 
Therefore we have the unique positive equilibrium 
$u_{\infty}$ with  $u_{\infty}^{\alpha} = u_{\infty}^{\beta}$ under the conservation law 
$\forall i \in L:= \{ i \in \{1,...,n\} | \alpha_i > \beta_i \}$ 
, $\forall j \in R := \{ j \in \{1,...,n\} | \alpha_j < \beta_j \}$.
\begin{equation} 
\begin{split}   \label{conservation law on steady state}
(\alpha_j-\beta_j)\int_{\Omega} \Tilde{u}_{i}(t,x) \ dx 
+ (\beta_i-\alpha_i)\int_{\Omega} \Tilde{u}_{j}(t,x) \ dx
= (\alpha_i-\beta_i) {u}_{i_{\infty}} 
+ (\beta_j-\alpha_j) {u}_{j_{\infty}}
:= M_{i,j}.
\end{split}
\end{equation}
We exclude the case when $\int_{\Omega} \Tilde{u}_{i}(t,x) \ dx = 0$, since in this situation the system degenerates to the heat equation and the solution converges to the boundary.

The method to prove local stability for $
\alpha_1 A_1 + ... +  \alpha_n A_n \rightleftharpoons \beta_1 A_1 + ... +  \beta_n A_n
$ is similar as in $A+2B\rightleftharpoons B+C$ case. We show the energy which consists of $L^2$ norm of $u$ and $L^2$ norm of $u_t$ is non-increasing by energy estimate and analysing the sign status for every $u_i$. Then the elliptic estimate can show the local stability for $u_i$ in $H^2$ sense.

In this paper we prove the local stability for the unique positive equilibrium $u_{\infty} = ({u}_{1_{\infty}},{u}_{2_{\infty}},...,{u}_{n_{\infty}})$, namely:

\begin{theorem}   \label{stable theorem in general}
For system \eqref{system n}, there exists small constant $\theta$ such that if the initial perturbation $u(x,0)$ satisfying 
$$
\sum\limits^{n}_{i=1} (\| \partial_t u_i(x,0)\|_2 + \|  u_i(x,0)\|_{\infty}) \leq \theta
$$
, then we have 
$$
\sum\limits^{n}_{i=1} \| u_i(x,t)\|_{H_2}  \lesssim 
e^{-lt} 
$$ 
where $l$ depends explicitly on $\alpha, \beta$ and $\theta$.
\end{theorem}

The above theorem will be proved in Section 3.

\section{Instability of boundary equilibria}

\subsection{Instability for $A+2B \rightleftharpoons B+C$ }

Since we want to show the instability at the boundary equilibrium $( a_{\infty} ,0,c_{\infty})$, we introduce three new variables as perturbation around the boundary equilibria.
\begin{equation}
a = \Tilde{a} - a_{\infty}, b = b, c = \Tilde{c} - c_{\infty}, u =(a,b,c)^{\intercal}
\end{equation}
Thus we have the conservation law for $(a,b,c)$;
\begin{equation}
\begin{split}  \label{conservation law for a,b,c}
& \int_{\Omega} a(t,x) \ dx+\int_{\Omega} c(t,x) \ dx = 0, 
 \int_{\Omega} b(t,x) \ dx+\int_{\Omega} c(t,x) \ dx = 0
\end{split}
\end{equation}
Note that
\begin{equation*}
\begin{split}
- \Tilde{a}\Tilde{b}^2+\Tilde{b}\Tilde{c} 
& = -(a+a_{\infty})b^2+b(c+c_{\infty})
\\& = b c_{\infty} + (bc - (a+a_{\infty})b^2)
\end{split}
\end{equation*}
Therefore we get the equations for $u$;
\begin{equation}
\begin{cases}     \label{equation for u}
a_t - d_a \Delta a = b {c}_{\infty} + (bc - (a+a_{\infty})b^2) &  x \in \Omega, t > 0 \\
b_t - d_b \Delta b  
= b {c}_{\infty} + (bc - (a+a_{\infty})b^2) &  x \in \Omega, t > 0 \\
c_t - d_c \Delta c 
= - b {c}_{\infty} - (bc - (a+a_{\infty})b^2)  &  x \in \Omega, t > 0 \\
\frac{\partial a}{\partial \mathbf{n}}= \frac{\partial b}{\partial \mathbf{n}}=\frac{\partial c}{\partial \mathbf{n}} =0 &  x \in \partial \Omega, t > 0\\
a(x,0)=a_0(x),b(x,0)=b_0(x),c(x,0)=c_0(x) &  x \in \Omega \\
\end{cases}
\end{equation}
It is convenient to express \eqref{equation for u} as
\begin{equation}  \label{equation for a and b and c}
u_t = L_1 u + N_1 (u),
\end{equation}
where $L_1 \coloneqq  
\left(
\begin{matrix}
   d_a \Delta  & c_{\infty} & 0 \\
   0 & d_b \Delta  + c_{\infty} & 0 \\
   0 & -c_{\infty} & d_c \Delta 
\end{matrix}
\right) $ and 
$N_1 (u) \coloneqq 
\begin{pmatrix}
    bc - (a+a_{\infty})b^2 \\
    bc - (a+a_{\infty})b^2 \\
    - bc + (a+a_{\infty})b^2 \\
\end{pmatrix}.
$

We cite Theorem 1.1 and Theorem 1.2 in Section 11.3 in \cite{Strauss.1992}.
For open and bounded the domain $\Omega$ with sufficient smooth boundary and the Neumann boundary condition, we denote the eigenvalues by  $\lambda_j$ and the eigenfunctions by $v_j(x)$. Thus
\begin{equation*}
\begin{cases}
- \Delta v_j(x) = \lambda_j v_j(x) &  x \in \Omega  \\
\frac{\partial v_j (x)}{\partial \mathbf{n}} = 0 &  x \in \partial  \Omega  \\
\end{cases}
\end{equation*}
Then we can number them in ascending order,
$$
0 = \lambda_1 < \lambda_2 \leq \lambda_3 \leq ...
$$
The first eigenfunction $v_1 (x)$ is a constant and the eigenfunctions forming a basis are complete in the $L_2$ sense.

Therefore the largest eigenvalue for Laplace operator is zero with the corresponding eigen-function is the constant function.

\begin{lemma}  \label{lemma1}
For the linear partial differential equations 
\begin{equation*}
\begin{cases}     
u_t = L_1 u &  x \in \Omega, t > 0 \\
\frac{\partial u}{\partial \mathbf{n}} =0 &  x \in \partial \Omega, t > 0,\\
\end{cases}
\end{equation*}
we have the following estimate 
$$
\|e^{t L_1} u_0 \|_{2} \leq 3 e^{ {c}_{\infty} t} \| u_0 \|_{2}
$$
\end{lemma} 
  
\begin{proof}
To get the eigenvalues $\lambda$ for $(d_b \Delta  + c_{\infty})b$ such that 
$$
(d_b \Delta  + c_{\infty})b = \lambda b \implies d_b \Delta b = (\lambda - c_{\infty}) b
$$
Since the largest eigenvalue for Laplace operator is zero, we have
$\lambda \leq c_{\infty}$.
And because the eigenfunctions forming a basis are complete in the $L_2$ sense, we can write initial data $b_0 \in L^2$ as 
$
b_0(x) = \sum\limits_{j} b_j v_j(x)   
$
in the $L_2$ sense and we get
$$
b(t,x) = e^{t L_1} b_0 = \sum\limits_{j} b_j e^{\lambda_j t} v_j(x)   
$$
also in the $L_2$ sense, therefore
\begin{equation*}
\begin{split}
 \| b(t,x) \|_{2} 
& = \| \sum\limits_{j} b_j e^{\lambda_j t} v_j(x) \|_2
 = \sum\limits_{j}  \|  e^{\lambda_j t} b_j v_j(x) \|_2
\\& \leq \sum\limits_{j}  \|  e^{c_{\infty} t} b_j v_j(x) \|_2 = e^{c_{\infty} t} \| \sum\limits_{j} b_j v_j(x) \|_2 = e^{c_{\infty} t} \| b_0 \|_{2}
\end{split}
\end{equation*}
Then for $a_t = d_a \Delta a + c_{\infty}b$, $c_t = d_c \Delta c - c_{\infty}b$, we multiply $a$ and $c$, integrate over domain $\Omega$ respectively and we get
\begin{equation*}
\begin{split}
& \frac{1}{2}\frac{d}{dt} \|a\|_{2}^{2} 
 = \int_{\Omega} d_a \Delta a \ a \ dx + \int_{\Omega} c_{\infty} b a \ dx
\\& = - \int_{\Omega} d_a | \nabla a |^{2}  dx + \int_{\Omega} c_{\infty} b a \ dx
 \leq c_{\infty} \|a\|_{2} \|b\|_{2} 
 \end{split}
\end{equation*}
\begin{equation*}
\begin{split}
\\&  \frac{1}{2}\frac{d}{dt} \|c\|_{2}^{2} 
  = \int_{\Omega} d_c \Delta c \ c \ dx - \int_{\Omega} c_{\infty} b c \ dx
\\& = - \int_{\Omega} d_c | \nabla c |^{2}  dx - \int_{\Omega} c_{\infty} b c \ dx
 \leq c_{\infty} \|c\|_{2} \|b\|_{2}
\end{split}
\end{equation*}
which implies
\begin{equation*}
\begin{split}
& \frac{d}{dt} \|a\|_{2} \leq c_{\infty} \|b\|_{2} \leq c_{\infty} e^{c_{\infty} t} \| b_0 \|_{2},
, \ \frac{d}{dt} \|c\|_{2} \leq c_{\infty} \|b\|_{2} \leq c_{\infty} e^{c_{\infty} t} \| b_0 \|_{2}.
 \end{split}
\end{equation*}
Therefore we have 
\begin{equation}
\begin{split}
& \|a(t,x) \|_{2} \leq \| a_0 \|_{2} + e^{c_{\infty} t} \| b_0 \|_{2}
, \ \|c(t,x) \|_{2} \leq \| c_0 \|_{2} + e^{c_{\infty} t} \| b_0 \|_{2}
 \end{split}
\end{equation}

\end{proof}

In order to use the elliptic estimate, we also need the following variables
$$
a_t = \Tilde{a}_t , b_t = \Tilde{b}_t, c_t = \Tilde{c}_t , u_t =(a_t,b_t,c_t)^{\intercal}
$$
Taking the time derivative on \eqref{equation for a and b and c}, we get
\begin{equation}
\begin{split}  \label{equation for a_t,b_t,c_t}
& u_{tt} = L_{2} u_t + N_{2}(u, u_t)
\end{split}
\end{equation}
where 
$N_{2}(u, u_t) \coloneqq \partial_t [N_1 (u)]
$ 
and $L_2 =  L_1$.

Now we define $y^{\intercal} = (u^{\intercal}, u^{\intercal}_t)$ and get the equation for $y$,
\begin{equation}
\begin{split}  \label{equation for y}
y_t 
= L y + N(y)
\end{split}
\end{equation}
where $L = \left(
\begin{matrix}
   L_1 & 0 \\
   0 & L_{2} 
\end{matrix}
\right) 
\ \text{and} \
N(y) = \begin{pmatrix}
    N_1 (u) \\
    N_{2} (u, u_t) \\
\end{pmatrix}$.

Considering Lemma \ref{lemma1} and $L$ is block diagonal matrix, we can get 
\begin{equation*}
\begin{split}
 \|e^{t L} y \| 
& = \|e^{t L_1} u \|_2 + \|e^{t L_2} u_t \|_2
\\& \leq 3 e^{{c}_{\infty}t} \|u\|_2 + 3 e^{{c}_{\infty}t} \|u_t \|_2 = 3 e^{{c}_{\infty}t} \|y \|
\end{split}
\end{equation*}
Therefore we have
\begin{equation}   \label{L estimate}
\|e^{t L}  \| = \sup_{\| y \| \leq \textcolor{blue}{1}} \frac{\|e^{t L} y \|}{ \| y \|}     
\leq 3 e^{{c}_{\infty} t}
\end{equation}
\\
In order to get the elliptic estimate, we cite the \text{Theorem 10.5} in \cite{ADN.1964}.

\textbf{Supplementary Condition on $L$.} $L(P,\Xi)$ is of even degree $2m$ (with respect to $\Xi$). For evev pair linearly independent real vectors $\Xi, \Xi^{\prime}$,, the polynomial $L(P,\Xi+ \tau\Xi^{\prime})$ in the complex variable $\tau$ has exactly $m$ roots with positive imaginary part.

In this condition, $P$ represents the points on the boundary $\partial \Omega$ of $\Omega$ with
$\Xi$ a tangent, and $\Xi^{\prime}$ the normal to $\partial \Omega$, at $P$.

\textbf{Complementing Boundary Condition.} For any $P \in \partial \Omega$
and any real, non-zero vector $\Xi$ tangent to $\partial \Omega$ at $P$, let us regard $M^{+}(P, \Xi, \tau) = \prod\limits_{h=1}^{m=3} (\tau - \tau_{h}^{+}(P,\Xi))$ where $\tau_{h}^{+}(P,\Xi)$ with $h=1,2,3$ are the $m$ roots (in $\tau$) with positive imaginary part of the characteristic equation $L(P,\Xi+ \tau\Vec{n}) = 0$.
and the elements of the matrix
$$
\sum\limits_{j=1}^{N}B^{\prime}_{hj}(P,\Xi+ \tau\Vec{n}) L^{jk}(P,\Xi+ \tau\Vec{n})
$$
as polynomials in the indeterminate $\tau$ where
$L^{jk}(P,\Xi+ \tau\Vec{n})$ is the matrix adjoint to $(l^{\prime}_{ij}(P,\Xi+ \tau\Vec{n}))$.
The definition of $l^{\prime}_{ij}$ will be shown in the following Theorem.
The rows of the latter matrix are required to be linearly independent modulo $M^{+}(P, \Xi, \tau)$, i.e.,
$$
\sum\limits_{h=1}^{m}C_h  \sum\limits_{j=1}^{N}B^{\prime}_{hj}L^{jk} \equiv 0 \ (\text{mod} \ M^{+})
$$
only if the constants $C_h$ are all zero.

\begin{theorem}  \label{nirenberg}
For the elliptic systems of partial differential equations  
$$
\sum\limits_{j=1}^{N} l_{ij}(P, \partial) u_j(P) = F_i(P), \ i=1,...,N
$$
where the $l_{ij}(P, \partial)$, linear differential operators, are polynomials in $\partial = \{\partial_{x_1},...,\partial_{x_{n+1}} \}$ with coefficients depending on $P$ over some domain 
$\Omega$ in $x_1,..., x_{n+1}$-space. The orders of these operators are assumed to depend on two systems of
integer weights, $s_1,..., s_N$ and $t_1,..., t_N$, attached to the equations and to the
unknowns, respectively, $s_i$ corresponding to the $i$-th equation and $t_j$ to the $j$-th dependent variable $u_j$.
The manner of the dependence is expressed by the inequality
$$
\deg l_{ij}(P, \Xi) \leq s_i + t_j   \   i,j = 1,...,N
$$
“deg” referring of course to the degree in $\Xi$.

If $L = \det (l^{\prime}_{ij}(P))$ where $l^{\prime}_{ij} (P)$ consists of the terms in $l_{ij}(P)$ which are just of the order
$s_i + t_j$ (the leading part with the highest order)
satisfies the supplementary condition and the boundary conditions are complementing  
$$
\sum\limits_{j=1}^{N} B_{hj}(P, \partial) u_j(P) = \phi_h(P) \ on \ \partial \Omega, \ h=1,...,m
$$
in terms of given polynomials in $\Xi$, $B_{hj}(P, \Xi)$, with complex coefficients depending
on $P$ with $m = \frac{1}{2} \deg(L(P)) > 0$. The orders of the boundary operators depend on two systems of integer weights, in this case the system $t_1,..., t_N$,
already attached to the dependent variables and a new system $r_1,..., r_m$ of which
$r_h$ pertains to the $h$-th boundary condition. The exact dependence is that expressed by the inequality
$$
\deg B_{hj}(P, \Xi) \leq r_h + t_j   \   h=1,...,m,  j = 1,...,N
$$
A constant K exists such that, if $\| u_j \|_{l_1 + t_j}$ for j=1,...,N, 
then 
for a given integer $l \geq l_1$,
$\| u_j \|_{l + t_j}$ also is finite, and
$$
\| u_j \|_{l + t_j} \leq K( \sum\limits_{i} \| F_i \|_{l - s_i} + \sum\limits_{h} \| \phi_h \|_{l - r_h - 1/p} + \sum\limits_{j} \| u_j \|_{0})
$$
where $\| \cdot \|_j = \| \cdot \|_{H_j}$
and K is dependent on the domain and the modulus
of continuity of the leading coeflcients in the $l_{ij}$.

\end{theorem}   

From the above Theorem \ref{nirenberg}, we get the following lemma.
\begin{lemma}
For the system 
$
u_t = L_1 u + N_1 (u)
$ in \eqref{equation for a and b and c} with Neumann boundary condition $\frac{\partial u}{\partial \mathbf{n}}|_{\partial \Omega} =0 $, we have the following elliptic estimate
$$
\|u_i\|_{H_2} \lesssim \|N_1 (u)\|_2 + \|u\|_2 + \|u_t\|_2.
$$
\end{lemma}

\begin{proof}

We first need to check whether the system satisfies the conditions in  Theorem \ref{nirenberg}.
Now we rewrite the system \eqref{equation for a and b and c} by putting $u_t$ to the right side.

We set $s_i = 0, t_j = 2$ with $1 \leq i,j \leq 3$. Therefore we get 
$$
L(P,\Xi) = \det (l^{\prime}_{ij}(P,\Xi)) = d_ad_bd_c(\xi_1^2+\xi_2^2+\xi_3^2)^3
$$
where $\Xi = (\xi_1,\xi_2,\xi_3)$.
It's obvious to see $L \neq 0$ for real $\Xi \neq 0$ which implies it is the elliptic system.
Next we check the supplement condition on operator $L$. $L(P,\Xi)$ is of the even degree $2m$ with $m=3$. Then for every pair of linearly independent real vectors $\Xi, \Xi^{\prime}$, we have 
$$
L(P,\Xi+ \tau\Xi^{\prime}) = d_ad_bd_c((\xi_1+\tau\xi_1^{\prime})^2+(\xi_2+\tau\xi_2^{\prime})^2+(\xi_3+\tau\xi_3^{\prime})^2)^3
$$
The above polynomial has exactly $m=3$ roots with positive imaginary roots since any real number can't be the root because of the linear independence and symmetric of the polynomial. We can also pick sufficient large $A$ such that
$$
A^{-1}|\Xi|^{2m} \leq |L(P,\Xi)| \leq A|\Xi|^{2m}
$$
to show the system is uniform elliptic.

Next we need to check whether Neumann boundary condition is complementing .
Since we have Neumann boundary condition which means 
$$
(n_1 \cdot \partial_{1} + n_2 \cdot \partial_{2} + n_3 \cdot \partial_{3})v_i = 0 \ \text{for} \ i=1,2,3.
$$
Then we set $r_h = -1$ with $h=1,2,3$. 

Here we set $\Xi$ be any tangent to $\partial \Omega$ and $P \in \partial \Omega$.
Therefore $B^{\prime}_{hj}(P,\Xi) = n_1 \cdot \xi_{1} + n_2 \cdot \xi + n_3 \cdot \xi_{3}$ if $h=j$.
Since we know $L(P,\Xi+ \tau\Vec{n})=0$ has three roots with positive imaginary part $\tau_{h}^{+}(P,\Xi)$ with $h=1,2,3$. We set 
$$
M^{+}(P,\Xi,\tau) = \prod\limits_{h=1}^{m=3} (\tau - \tau_{h}^{+}(P,\Xi))
$$
And let $(L^{jk}(P,\Xi+ \tau\Vec{n}))$ denote the matrix ad-joint to $(l^{\prime}_{ij}(P,\Xi+ \tau\Vec{n}))$. Then we have $(L^{jk}(P,\Xi+ \tau\Vec{n})) = \mathrm{diag}\{d_b d_c,d_a d_c,d_a d_b\} \cdot (\xi_1^2+\xi_2^2+\xi_3^2)^2$ which is also a diagonal matrix.
Thus we get
\begin{equation*}
\begin{split}
& \sum\limits_{h=1}^{m=3}C_h  \sum\limits_{j=1}^{N=3}B^{\prime}_{hj}L^{jk}(P,\Xi+ \tau\Vec{n}) \ \text{for} \ k=1,2,3
\\& = \Tilde{C_k} (n_1 \cdot (\xi_{1}+\tau n_1) + n_2 \cdot (\xi+\tau n_2) + n_3 \cdot (\xi_{3}+\tau n_3))
\\&
((\xi_{1}+\tau n_1)^2+(\xi_{2}+\tau n_2)^2+(\xi_{3}+\tau n_3)^2)^2 \equiv 0 \ (\text{mod} \ M^{+})
\end{split} 
\end{equation*}
\\
Only if $\Vec{n} \parallel (\xi_{1}+\tau n_1,\xi_{2}+\tau n_2,\xi_{3}+\tau n_3)$ or $\{C_k\}$ are all zero. It's obvious to see that Neumann boundary conditions satisfy the complementing boundary condition.
Then Theorem \ref{nirenberg} shows that with $l_1 = \max(0,r_h+1) = 0$, if $\|u_i\|_{H_2}$ are all finite, pick $l = l_1$, then for $i=1,2,3$, we have
\begin{equation} \label{elliptic estimate for instability}
\|u_i\|_{H_2} \leq K (\| N_u (u) \|_2
+ \sum\limits_{i=1}^{3}\|\partial_t u_i\|_2
+ \sum\limits_{i=1}^{3}\|u_i\|_2)
\end{equation}
where $K$ is a constant depends on origin equation and bounded domain.
\end{proof}

Now we we start proving our main theorem, Theorem \ref{instable theorem}.
First we show the existence of $y_0$ and the corresponding constant $C_P$.

\begin{lemma}
If $\| y_0 \| = 1$, $\int_{\Omega} b_0 dx \neq 0$ $(b_0 \geq 0)$ and $\interleave  y_0 \interleave < \infty$, there exists $C_p > 0$ such that , there exists $C_p > 0$ such that 
$$
\|e^{tL} y_0\| \geq C_P e^{c_{\infty} t}
$$
\end{lemma}

\begin{proof}     

In our case, the conservation law and $\int_{\Omega} b_0 dx \neq 0$ imply
\begin{equation}  \label{initial intgration}
\int_{\Omega} b_0(x) dx = \int_{\Omega} a_0(x) dx = - \int_{\Omega} c_0(x) dx > 0
\end{equation}
Taking the integration over the domain $\Omega$ on first linear part $u_t = L_1 u $, we get 
\begin{equation}
 \frac{d}{dt} \int_{\Omega} b(t,x) dx = c_{\infty}  \int_{\Omega} b(t,x) dx
\end{equation}
this implies 
\begin{equation}
\int_{\Omega} b(t,x) dx = e^{c_{\infty} t} \int_{\Omega} b_0(x) dx
\end{equation}
Similarly from $u_t = L_1 u $, we get equations for $a$ and $c$
\begin{equation}   \label{a and c time diff}
 \frac{d}{dt} \int_{\Omega} a(t,x) dx = c_{\infty}  \int_{\Omega} b(t,x) dx,
\ \frac{d}{dt} \int_{\Omega} c(t,x) dx = - c_{\infty}  \int_{\Omega} b(t,x) dx
\end{equation}
From \eqref{initial intgration} and \eqref{a and c time diff}, we have
\begin{equation}
  \int_{\Omega} a(t,x) dx = e^{c_{\infty} t} \int_{\Omega} b_0(x) dx 
, \  \int_{\Omega} c(t,x) dx = - e^{c_{\infty} t} \int_{\Omega} b_0(x) dx
\end{equation}
this implies 
\begin{equation} \label{first part initial estimate}
\|e^{t L_1} u_0\|_2 \geq 3 \bar{b_0} e^{c_{\infty} t}
\end{equation}
where $\bar{b_0} := \int_{\Omega} b_0 dx$.
Again by the conservation law, the second part $u_{tt} = L_{2} u_t$ shows 
\begin{equation}
  \int_{\Omega} b_t(t,x) dx = \frac{d}{dt} \int_{\Omega} b(t,x) dx = c_{\infty} e^{c_{\infty} t} \int_{\Omega} b_0(x) dx 
\end{equation}
Also by the conservation law \eqref{conservation law for a,b,c} and \eqref{a and c time diff}, we have
\begin{equation}
\begin{split}
&  \int_{\Omega} a_t(t,x) dx = \frac{d}{dt} \int_{\Omega} a(t,x) dx = c_{\infty} e^{c_{\infty} t} \int_{\Omega} b_0(x) dx 
\\& \int_{\Omega} c_t(t,x) dx = \frac{d}{dt} \int_{\Omega} c(t,x) dx = - c_{\infty} e^{c_{\infty} t} \int_{\Omega} b_0(x) dx 
\end{split}
\end{equation}
this again implies,
\begin{equation} \label{second part initial estimate}
\|e^{t L_2} u_t \|_2 \geq 3 c_{\infty} \bar{b_0} e^{c_{\infty} t}
\end{equation}
From \eqref{first part initial estimate} and \eqref{second part initial estimate}, we can have the following
$$
\|e^{tL} y_0\| \geq C_P e^{c_{\infty} t}
$$
where $C_P = 3 (c_{\infty}+1) \bar{b_0}$.
\end{proof}

Then we do the estimate on the non-linear part $N(y)$ in the norm of
$\| \cdot \|$.

\begin{lemma}
$$     \label{N estimate}
\|N(y)\| \lesssim \interleave y \interleave^2 + \interleave y \interleave^3 
$$
\end{lemma}

\begin{proof}
\begin{equation}
\begin{split}
 \|N(y)\| 
&= \|N_1 (u)\|_2 + \|N_{2} (u, u_t)\|_2 
\\&= 3\|bc - (a+a_{\infty})b^2\|_2 + 3\|(bc - (a+a_{\infty})b^2)_t\|_2  
\end{split}
\end{equation}
By using Sobolev embedding inequality 
$
\| y \|_{\infty} \leq C_{SI} \cdot \| y \|_{H_2}
$, we can control the right hand side by norm $\interleave  \cdot \interleave $.
For the $N_1 (u)$ part,
\begin{equation*}
\begin{split} 
& \|bc - (a+a_{\infty})b^2\|_2 
\leq \|bc\|_2 + \|ab^2\|_2 + \| \Tilde{a}_{\infty}b^2 \|_2
\\& \leq \|b\|_{\infty} (\|c\|_2 + \| \Tilde{a}_{\infty}b \|_2 + \|b\|_{\infty} \|a\|_2 )
\\& \leq C_{SI}(1+ a_{\infty}) \interleave y \interleave^2 + C^{2}_{SI} \interleave y \interleave^3
\end{split}
\end{equation*}
For the $N_2 (u, u_t)$ part,
\begin{equation*}
\begin{split} 
& \|(bc - (a+a_{\infty})b^2)_t\|_2 
\leq \|(bc)_t\|_2 + \|(ab^2)_t\|_2 + \| (\Tilde{a}_{\infty}b^2)_t \|_2
\\& \leq (\|b\|_{\infty} \|c_t\|_2  + \|c\|_{\infty} \|b_t\|_2)
+ (\|b\|^2_{\infty} \|a_t\|_2 + 2\|a\|_{\infty}\|b\|_{\infty} \|b_t\|_2) + 2\Tilde{a}_{\infty}\|b\|_{\infty} \| b_t \|_2
\\& \leq 2 C_{SI}(1+ a_{\infty}) \interleave y \interleave^2 + 3 C^{2}_{SI} \interleave y \interleave^3
\end{split}
\end{equation*}
Combining the above two parts, we get the following 
\begin{equation}  
\|N(y)\| \leq C_N ( \interleave y \interleave^2 + \interleave y \interleave^3 )
\end{equation}
for all $y$ and $\interleave y \interleave \leq \infty$ and constant  
$C_N = \max{\{ 9 C_{SI}(1+ a_{\infty}), 12 C^{2}_{SI}} \}$.

\end{proof}

Next we do the estimate on $u$ and $u_t$ in the norm of $\interleave  \cdot \interleave$. 

\begin{lemma} \label{sigma estimate}
Suppose $\interleave y \interleave < \sigma$ 
and $\sigma$ is sufficiently small
with 
$\|a, b, c\|_{\infty} < C_{SI} \cdot \sigma$
such that
$\|b\|_{\infty} (1 + \| a\|_{\infty} + a_{\infty}) < \min(c_{\infty}, 1)$ and $ \|b\|_{\infty} + \|b\|^2_{\infty} + \|c\|_{\infty} + 2(\|a\|_{\infty}+a_{\infty}) \|b\|_{\infty} < \min(c_{\infty}, 1) $, we have the following estimate 
$$
\interleave y \interleave^2 \lesssim \int^{t}_{0} \|y\|^2 \ ds + \|y_0\|^2
$$
\end{lemma}

\begin{proof}
Given $\interleave y \interleave < \sigma$ is sufficiently small, we have the smallness of $\|a, b, c\|_{\infty} < C_{SI} \cdot \sigma$ by Sobolev embedding inequality.
Recall the equations \eqref{equation for u}, we multiply $a, b, c$ respectively and do the integral over the domain $\Omega$ and get 
\begin{equation}
\begin{split}  \label{estimate for u}
& \frac{1}{2} \frac{d}{dt}(\|a\|^{2}_{2} + \|b\|^{2}_{2} + \|c\|^{2}_{2}) + (d_a \| \nabla a\|^{2}_{2} + d_b \| \nabla b\|^{2}_{2} + d_c \| \nabla c\|^{2}_{2})
\\& = \int_{\Omega} (a+b-c) b\Tilde{c}_{\infty} + (a+b-c) (bc - (a+a_{\infty})b^2) \ dx
\\& \leq c_{\infty} \|b\|_{2}(\|a\|_{2} + \|b\|_{2} + \|c\|_{2}) + \int_{\Omega} (a+b-c) (bc - (a+a_{\infty})b^2) \ dx
\\& \leq c_{\infty} \|u\|_{2}^2 + \|b\|_{\infty} (1 + \| a\|_{\infty} + a_{\infty}) \|u\|_{2}^2 \lesssim \|u\|_{2}^2
\end{split} 
\end{equation}
Next on the equations 
\eqref{equation for a_t,b_t,c_t}, we multiply $a_t, b_t, c_t$ respectively and do the integral over the domain $\Omega$ again and get 
\begin{equation*}
\begin{split} 
& \frac{1}{2} \frac{d}{dt}(\|a_t\|^{2}_{2} + \|b_t\|^{2}_{2} + \|c_t\|^{2}_{2}) + (d_a \| \nabla a_t\|^{2}_{2} + d_b \| \nabla b_t\|^{2}_{2} + d_c \| \nabla c_t\|^{2}_{2})
\\& = \int_{\Omega} (a_t + b_t - c_t) b_t c_{\infty} + (a_t +b_t -c_t)(bc - (a+a_{\infty})b^2)_t \ dx
\\& = \int_{\Omega} (a_t + b_t - c_t) b_t c_{\infty} + (a_t +b_t -c_t)(b_t c + b c_t- a_tb^2 
- 2(a+a_{\infty})bb_t) \ dx
\\& \leq c_{\infty} \|b_t \|_{2}(\|a_t\|_{2} + \|b_t\|_{2} + \|c_t\|_{2}) + (\|a_t\|_{2} + \|b_t\|_{2} + \|c_t\|_{2}) (\|b_t\|_{2} \|c\|_{\infty} \\& + \|c_t\|_{2} \|b\|_{\infty} + \|a_t\|_{2} \|b\|^2_{\infty} + 2(\|a\|_{\infty}+a_{\infty}) \|b\|_{\infty} \|b_t\|_{2})
\\& \leq C_{u_t} (\|a_t\|_{2} + \|b_t\|_{2} + \|c_t\|_{2})^2
\end{split} 
\end{equation*}
where constant $C_{u_t} = 3 [ c_{\infty} + \|b\|_{\infty} + \|b\|^2_{\infty} + \|c\|_{\infty} + 2(\|a\|_{\infty}+a_{\infty}) \|b\|_{\infty}]$.

Then we get
\begin{equation}
\begin{split}  \label{estimate for u_t}
\frac{d}{dt}(\|a_t\|^{2}_{2} + \|b_t\|^{2}_{2} + \|c_t\|^{2}_{2})
\leq 2 C_{u_t} (\|a_t\|_{2} + \|b_t\|_{2} + \|c_t\|_{2})^2
\end{split} 
\end{equation}

Recall the elliptic estimate \eqref{elliptic estimate for instability}
$$
\|u_i\|_{H_2} \lesssim \|bc - (a+a_{\infty})b^2\|_2 + \|u\|_2+\|u_t\|_2, u = (a,b,c)^{\intercal}
$$
Combining this with \eqref{estimate for u}, \eqref{estimate for u_t}, we get the $H_2$ estimate for $u$ which is the first part of $\interleave \cdot \interleave$ norm
\begin{equation}
\begin{split}  \label{H2 estimate}
& \|u_i\|^2_{H_2} \lesssim \|bc - (a+a_{\infty})b^2\|^2_2 + \|u\|^2_2+\|u_t\|^2_2 
\\& \leq \|b\|^2_{\infty}(1 + \|a\|_{\infty} + a_{\infty})^2 \|u\|^2_2 + \|u\|^2_2 + \|u_t\|^2_2  
\lesssim  \|u\|^2_2 + \|u_t\|^2_2
\\& \lesssim \int^{t}_{0} \|b\|_{2} (\|a\|_{2} + \|b\|_{2} + \|c\|_{2}) \ ds + \|a_0\|^{2}_{2} + \|b_0\|^{2}_{2} + \|c_0\|^{2}_{2}
\\& + \int^{t}_{0} (\|a_t\|_{2} + \|b_t\|_{2} + \|c_t\|_{2})^2 \ ds + \|a_t(0)\|^{2}_{2} + \|b_t(0)\|^{2}_{2} + \|c_t(0)\|^{2}_{2}
\\& \lesssim \int^{t}_{0} \|y\|^2 \ ds + \|y_0\|^2
\end{split} 
\end{equation}
where the above inequalities hold because $\interleave y \interleave \leq \sigma$.
Again from \eqref{estimate for u_t}, we get the $L_2$ estimate for $u_t$ which is the second part of $\interleave \cdot \interleave$ norm
\begin{equation}
\begin{split}  \label{u_t L_2 estimate}
& \|u_t\|^2_2 \lesssim \|a_t\|^{2}_{2} + \|b_t\|^{2}_{2} + \|c_t\|^{2}_{2}
\\& \leq \int^{t}_{0} (\|a_t\|_{2} + \|b_t\|_{2} + \|c_t\|_{2})^2 \ ds + \|a_t(0)\|^{2}_{2} + \|b_t(0)\|^{2}_{2} + \|c_t(0)\|^{2}_{2}
\\& \lesssim \int^{t}_{0} \|y\|^2 \ ds + \|y_0\|^2
\end{split} 
\end{equation}
\end{proof}

Finally, we proof Theorem \ref{instable theorem} with all above lemma. The proof is based on the argument of \cite{Guo-Strauss}.
\begin{proof}

Now we denote 
\begin{equation*}
\begin{split}
& T^{\delta} = \frac{1}{c_{\infty}} \log \frac{\theta_0}{\delta} 
\\& T^{*} = \sup_{t} \{\interleave y \interleave < \sigma \}
\\& T^{**} = \sup_{t} \{\|y\| \leq 2 \delta e^{c_{\infty} t} \|y_0\| \}
\end{split}
\end{equation*}
For $t \leq \min \{T^{\delta}, T^{*}, T^{**} \}$, we can get from \eqref{H2 estimate} and \eqref{u_t L_2 estimate}, and consider a family of initial data $y^{\delta}(0) = \delta y_0$ with $\| y_0 \| = 1$ and $\interleave y_0 \interleave < \infty$,
$$  
\interleave y \interleave^2 \lesssim \int^{t}_{0} \|y\|^2 \ ds + \delta^2 \|y_0\|^2
 \lesssim \|y_0\|^2 ( \delta^2 e^{2\Tilde{c}_{\infty} t} + \delta^2)
$$
which implies 
\begin{equation}
\begin{split}   \label{interleave norm estimate}
& \interleave y \interleave \lesssim \|y_0\| ( \delta e^{c_{\infty} t} + \delta ) \lesssim \delta e^{c_{\infty} t}
\end{split} 
\end{equation}
Then there exists the constant $C_1$ such that
$$
\interleave y \interleave \leq C_1 \delta e^{c_{\infty} t}
$$
Appling the Duhamel principle to $y_t = L y + N(y)$, we have 
\begin{equation}
\begin{split}  
 \|y(t) - \delta e^{ L t} y_0\| 
& = \|\int^{t}_{0} e^{L(t-\tau)}N(y(\tau)) \ d\tau \|
\\& \lesssim \int^{t}_{0} e^{ c_{\infty} (t-\tau)} \|N(y(\tau))\| \ d\tau 
\\& \lesssim \int^{t}_{0} e^{ c_{\infty} (t-\tau)} (\interleave y  \interleave^2 + \interleave y \interleave^3) \ d\tau
\\& \lesssim \int^{t}_{0} e^{ c_{\infty} (t-\tau)} (\delta^2 e^{2\Tilde{c}_{\infty} \textcolor{blue}{\tau}} + \delta^3 e^{3\Tilde{c}_{\infty} \textcolor{blue}{\tau}}) \ d\tau 
\\& \lesssim \delta^2 e^{2\Tilde{c}_{\infty} t} + \delta^3 e^{3\Tilde{c}_{\infty} t}
\end{split}
\end{equation}
where the first inequality holds by \eqref{L estimate}, the second inequality holds by \text{by Lemma \ref{N estimate}} and the third inequality holds by \eqref{interleave norm estimate}.
\\
Then there exists the constant $C_2$ such that
$$
\|y(t) - \delta e^{ L t} y_0\|  \leq C_2 (\delta^2 e^{2\Tilde{c}_{\infty} t} + \delta^3 e^{3\Tilde{c}_{\infty} t})
$$
In order to find the escape time, it suffices to show that 
$$
{ \min \left\{T^{\delta}, T^{*}, T^{* *}\right\}=T^{\delta}} 
$$
by fixing  $\theta_{0}$ small enough. Set
$$
\theta_{0} = \min \{\frac{\sigma}{C_1}, \frac{1}{2 C_2}, \frac{C_p}{4}, \sqrt{\frac{C_p}{4}}\}
$$
On the one hand, if $ T^{*} < T^{\delta} $ is the smallest, then  for $ 0 \leq t \leq T^{*} $,
\begin{equation*}
\interleave y(T^{*}) \interleave
\leq C_1 \delta e^{c_{\infty} T^{*}}
< C_1 \delta e^{c_{\infty} T^{\delta}} = C_1 \theta_0
< \sigma
\end{equation*}
which is a contradiction to the definition of $T^{*}$. On the other hand, if $T^{**} < T^{\delta}$ is the
smallest, then we have 
\begin{equation*}
\begin{split}
\| y(T^{**}) \|
& \leq \delta e^{ c_{\infty} T^{**}} \| y_0 \|
+ C_2 (\delta^2 e^{2\Tilde{c}_{\infty} T^{**}} + \delta^3 e^{3\Tilde{c}_{\infty} T^{**}})
\\& < \delta e^{ c_{\infty} T^{**}}
+ C_2 (\delta e^{c_{\infty} T^{**}} \theta_0 + \delta e^{c_{\infty} T^{**}} \theta^2_0)
< 2 \delta e^{c_{\infty} t}
\end{split}
\end{equation*}
which is a contradiction to the definition of $T^{**}$.

Moreover, if there exists a constant $C_p$ such that 
$$
\|e^{tL} y_0\| \geq C_p e^{c_{\infty} t},
$$ 
then at the escape time $t = T^{\delta}$, we have the following estimate
$$
\| \delta e^{L T^{\delta}} y_0\| \geq C_p \delta e^{c_{\infty} T^{\delta}} = C_p \theta_0,
$$
where the non-linear term is
$$
\delta^2 e^{2\Tilde{c}_{\infty} T^{\delta}} + \delta^3 e^{3\Tilde{c}_{\infty} T^{\delta}} = \theta^2_0 + \theta^3_0
$$
then 
$$
\|y(T^{\delta})\| \geq \tau_0 = \frac{1}{2} C_p \theta_0 > 0
$$
which depends explicitly on 
$\sigma$, $C_p, c_{\infty}, y_0$ and is independent of $\delta$.

\end{proof}

Therefore we conclude the local instability for $\delta y_0$
as long as $\| y_0 \| = 1$, $\int_{\Omega} b_0 dx \neq 0$ and $\interleave  y_0 \interleave < \infty$ and sufficient small $\delta$.

\begin{remark}
If the initial data $\int_{\Omega} b_0 dx = 0$, this means $b \equiv 0, \forall x \in \Omega,t >0$  and $R(u) \equiv 0, \forall x \in \Omega,t >0$ which implies the equations for $a$ and $c$ coincide with the heat equation. Therefore, in this case the system will converge to the accessible boundary equilibria $( a_{\infty} ,0,c_{\infty})$.
\end{remark}

\subsection{Instability in generalized case}

Here we indicate how to adapt the above analysis to get instability result for the following generalized case
$$
A_1 + ... + A_l +2B \rightleftharpoons B + C_1 + ... + C_r
$$
The corresponding reaction-diffusion system is

\begin{equation}
\begin{cases} 
\partial_t \Tilde{a_i} - d_i \Delta \Tilde{a_i}= - \Tilde{b}^2 \prod \Tilde{a_i}+\Tilde{b}  \prod \Tilde{c_j} & i=1,...,l, x \in \Omega, t > 0 \\
\partial_t \Tilde{b} - d_b \Delta \Tilde{b}=- \Tilde{b}^2 \prod \Tilde{a_i} +  \Tilde{b}  \prod \Tilde{c_j}  &  x \in \Omega, t > 0 \\
\partial_t \Tilde{c_j} - d_j \Delta \Tilde{c_j}= \Tilde{b}^2 \prod \Tilde{a_i} -  \Tilde{b}  \prod \Tilde{c_j}  &  j=1,...,r, x \in \Omega, t > 0 \\
\nabla{\Tilde{a}_i} \cdot n=\nabla{\Tilde{b}} \cdot n=\nabla{\Tilde{c}_j} \cdot n=0 &  x \in \partial \Omega, t > 0\\
\Tilde{a_i}(x,0)=\Tilde{a}_{i,0}(x),\Tilde{b}(x,0)=\Tilde{b}_0(x),\Tilde{c_j}(x,0)=\Tilde{c}_{j,0}(x) &  x \in \Omega \\
\end{cases}
\end{equation}
\\
For this reaction system, we have the following conservation laws;
\begin{equation}
\begin{split}  \label{conservation law general case}
& \int_{\Omega} \Tilde{a_i} \ dx+\int_{\Omega} \Tilde{c_j} \ dx=\int_{\Omega}\Tilde{a}_{i,0}(x) \ dx+\int_{\Omega} \Tilde{c}_{j,0}(x) \ dx := M_{1,ij}
\\& \int_{\Omega} \Tilde{b_i} \ dx+\int_{\Omega} \Tilde{c_j} \ dx=\int_{\Omega}\Tilde{b}_{i,0}(x) \ dx+\int_{\Omega} \Tilde{c}_{i,0}(x) \ dx 
:= M_{2,ij}
\end{split}
\end{equation}

Again we are interested in the {\em accessible boundary equilibrium} of a reaction network, as long as $M_{1,ij} > M_{2,ij} \ i=1,...,l,j=1,...,r$
there are two types of  equilibria following the conservation laws and we name 
 $( {a}_{i,\infty} , 0 , {c}_{j,\infty})$ as unique accessible boundary equilibria which follows  \eqref{conservation law general case},
$$
{a}_{i,\infty} + {c}_{j,\infty} = M_{1,ij}, {b}_{i,\infty} + {c}_{j,\infty} = M_{2,ij}
$$

Similarly we introduce new variables as perturbation  around the boundary equilibrium
$$
a_i = \Tilde{a_i} - {a}_{i,\infty}, b = b, c_j = \Tilde{c_j} - {c}_{j,\infty}, u =(a_i,b,c_j)^{\intercal}
$$
Then we can get the equation for $a_i$, $b$ and $c_j$ with $i=1,...,l$, $j=1,...,r$
\begin{equation}
\begin{cases}     \label{equation for a and b and c general case}
\partial_t a_i - d_i \Delta a_i =  b  \prod {c}_{j,\infty}  + N(a_i,b,c_j) &  x \in \Omega, t > 0 \\
\partial_t b - d_b \Delta b  
= b  \prod {c}_{j,\infty}  + N(a_i,b,c_j) &  x \in \Omega, t > 0 \\
\partial_t c_j - d_j \Delta c_j 
= - b  \prod {c}_{j,\infty}  - N(a_i,b,c_j)  &  x \in \Omega, t > 0 \\
\frac{\partial a}{\partial \mathbf{n}}= \frac{\partial b}{\partial \mathbf{n}}=\frac{\partial c}{\partial \mathbf{n}} =0 &  x \in \partial \Omega, t > 0\\
{a_i}(x,0)= {a}_{i,0}(x), {b}(x,0)= {b}_0(x {c_j}(x,0)= {c}_{j,0}(x) &  x \in \Omega \\
\end{cases}
\end{equation}
where $N(a_i,b,c_j) = - b^2 \prod (a_i + \Tilde{a}_{i,\infty}) + b \prod (c_j + {c}_{j,\infty}) - b  \prod {c}_{j,\infty} $.

Again we can express \eqref{equation for a and b and c general case} as
$$
u_t = L_1 u + N_u (u)
$$
where $L_1 \coloneqq  
\left(
\begin{matrix}
   d_i \Delta  & \prod {c}_{j,\infty} & 0 \\
   0 & d_b \Delta  + \prod {c}_{j,\infty} & 0 \\
   0 & -\prod {c}_{j,\infty} & d_j \Delta 
\end{matrix}
\right) $ and 
$N_{1}(u) \coloneqq 
\begin{pmatrix}
    N(a_i,b,c_j) \\
    N(a_i,b,c_j) \\
    - N(a_i,b,c_j) \\
\end{pmatrix}.
$

Similarly we can get the largest eigenvalue for $L_1$ is $\prod {c}_{j,\infty} > 0$, then we can get
$$
\|e^{t L_1} u_0 \|_{2} \leq e^{\prod {c}_{j,\infty} t} \|u_0 \|_{2}
$$
which implies
$$
\|e^{t L_1} \|_{2} \leq e^{\prod {c}_{j,\infty} t}.
$$

In order to use the elliptic estimate, we also need the following variables
$$
a_t = \Tilde{a}_t , b_t = \Tilde{b}_t, c_t = \Tilde{c}_t , u_t =(a_t,b_t,c_t)^{\intercal}
$$
Taking the time derivative on \eqref{equation for a and b and c general case}, we get
\begin{equation}
\begin{split}  \label{equation for a_t,b_t,c_t in general}
& u_{tt} = L_{2} u_t + N_{2}(u, u_t)
\end{split}
\end{equation}
where 
$N_{2}(u, u_t) \coloneqq \partial_t [N_1 (u)]
$ 
and $L_2 =  L_1$.
Recall $y^{\intercal} = (u^{\intercal}, u^{\intercal}_t)$ and get the equation for $y$,
\begin{equation}
\begin{split}  \label{equation for y}
y_t 
= L y + N(y)
\end{split}
\end{equation}
where $L = \left(
\begin{matrix}
   L_1 & 0 \\
   0 & L_{2} 
\end{matrix}
\right) 
\ \text{and} \
N(y) = \begin{pmatrix}
    N_1 (u) \\
    N_{2} (u, u_t) \\
\end{pmatrix}$.
Again considering Lemma \ref{lemma1} and $L$ is block diagonal matrix, we can get 
$$
\|e^{tL} \| \leq e^{\prod {c}_{j,\infty} t} 
$$

Since the linear term $Ly$ dominates $N(y)$ term (or the right hand side) because of the smallness of $\interleave y \interleave $ and 
the assumption of $\| y_0 \| = 1$, $\int_{\Omega} b_0 dx \neq 0$ $(b_0 \geq 0)$ and $\interleave  y_0 \interleave < \infty$ and the conservation law \eqref{conservation law general case} also implies the existence of the constant $C_p > 0$ such that
$ \|e^{tL} y_0\| \geq C_P e^{c_{\infty} t} $,
we can use the similar analysis as above to get the local instability of the accessible boundary equilibria.

\section{Local stability for $
\alpha_1 A_1 + ... +  \alpha_n A_n \rightleftharpoons \beta_1 A_1 + ... +  \beta_n A_n
$}

To show the stability at the unique positive equilibria $u_{\infty}$, we again introduce the small perturbation $u_{i} =\Tilde{u}_{i} - u_{i_{\infty}}$ around the boundary equilibria.
Then we get the following equation for perturbation.
\begin{equation}
\begin{split}
    & \partial_t u_i - d_i \Delta u_i = (\beta_i - \alpha_i)(\prod(u_{i} + u_{i_{\infty}} )^{\alpha_i} - \prod(u_{i} + u_{i_{\infty}} )^{\beta_i})
    \\& = (\beta_i - \alpha_i) \prod(u_{i} + u_{i_{\infty}} )^{\gamma_i}
    (\prod(u_{i} + u_{i_{\infty}} )^{\alpha_i - \gamma_i} - \prod(u_{i} + u_{i_{\infty}} )^{\beta_i - \gamma_i})
\end{split}
\end{equation}
where $\gamma = (\gamma_1,...,\gamma_n)$ with $\gamma_i = \min \{ \alpha_i, \beta_i \}$. 
We also donate
$$
L := \{ i \in \{1,...,n\} | \alpha_i > \beta_i \}, R := \{ j \in \{1,...,n\} | \alpha_j < \beta_j \}
$$
$$
L_0 := \{ i_0 \in \{1,...,n\} | \alpha_{i_0} \neq 0 \}, R_0 := \{ j_0 \in \{1,...,n\} | \beta_{j_0} \neq 0 \}
$$
and we assume $L \neq \emptyset, R \neq \emptyset, L \cup R = \{1,2,...,n\}$ and $ L_0 \cap R_0 \neq \emptyset $. 
The last assumption means we don't consider the case where the system only has positive equilibrium since \cite{DFT2} has already shown the global convergence without boundary equilibrium.

Now we we start proving the main theorem, Theorem \ref{stable theorem in general}
in this section.
First we do the energy estimate on the system.

W.l.o.g we assume there exists $m$ such that $0 < m < n$ and $L = \{ 1,...,m \}$, $R = \{m+1,...,n\}$. Then we write the perturbation in the following way
 \begin{equation}
\begin{split}  \label{equation in general}
    & \partial_t u_i - d_i \Delta u_i 
    \\& = (\beta_i - \alpha_i) \prod(u_{i} + u_{i_{\infty}} )^{\gamma_i}
    (\prod(u_{i} + u_{i_{\infty}} )^{\alpha_i - \gamma_i} - \prod(u_{i} + u_{i_{\infty}} )^{\beta_i - \gamma_i})
    \\& = (\beta_i - \alpha_i) \prod(u_{i} + u_{i_{\infty}} )^{\gamma_i}
    \{ [ u_{{\infty}} ^{\alpha - \gamma} + \sum\limits^{m}_{i=1} (\alpha_i - \gamma_i) u_i \frac{u_{{\infty}} ^{\alpha - \gamma}}{u_{i_{\infty}}} + N_1(u,u_{\infty}) ]
    \\& - [ u_{\infty}^{\beta - \gamma} + \sum\limits^{n}_{j= m+1} (\beta_j - \gamma_j) u_j \frac{u_{{\infty}} ^{\beta - \gamma}}{u_{j_{\infty}}} + N_2(u,u_{\infty}) ] \}
    \\& = (\beta_i - \alpha_i) \prod(u_{i} + u_{i_{\infty}} )^{\gamma_i}
     [  \sum\limits^{m}_{i=1} (\alpha_i - \gamma_i) u_i \frac{u_{{\infty}} ^{\alpha - \gamma}}{u_{i_{\infty}}} 
 -  \sum\limits^{n}_{j= m+1} (\beta_j - \gamma_j) u_j \frac{u_{{\infty}} ^{\beta - \gamma}}{u_{j_{\infty}}} 
    \\& + N_1(u,u_{\infty}) - N_2(u,u_{\infty}) ] 
\end{split}
\end{equation}
where $ N_1(u,u_{\infty}) = \prod(u_{i} + u_{i_{\infty}} )^{\alpha_i - \gamma_i} - u_{{\infty}} ^{\alpha - \gamma} - \sum\limits^{m}_{i=1} (\alpha_i - \gamma_i) u_i \frac{u_{{\infty}} ^{\alpha - \gamma}}{u_{i_{\infty}}}$ and \\
$N_2(u,u_{\infty}) = \prod(u_{i} + u_{i_{\infty}} )^{\beta_i - \gamma_i} - u_{{\infty}} ^{\beta - \gamma} - \sum\limits^{n}_{j= m+1} (\beta_j - \gamma_j) u_j \frac{u_{{\infty}} ^{\beta - \gamma}}{u_{j_{\infty}}} $, both $N_1$ and $N_2$ are non-linear term w.r.t. $u_i$ and for simplicity we define $N := N_1 - N_2$.

Multiplying $ \frac{(\alpha_i - \gamma_i)}{(\alpha_i - \beta_i)}  \frac{u_{{\infty}} ^{\alpha - \gamma}}{u_{i_{\infty}}} u_i , \frac{(\beta_j - \gamma_j)}{(\beta_j - \alpha_j)} \frac{u_{{\infty}} ^{\beta - \gamma}}{u_{j_{\infty}}} u_j$ on \eqref{equation in general} respectively, then integrating over $\Omega$ by parts, we get the following

\begin{equation}
\begin{split}   \label{first part energy estimate for L_2 in general}
        & \frac{1}{2}\frac{d}{dt}( \sum\limits^{m}_{i=1} (\frac{(\alpha_i - \gamma_i)}{(\alpha - \beta_i)}  \frac{u_{{\infty}} ^{\alpha - \gamma}}{u_{i_{\infty}}} \|u_i\|^2_2
 +  \sum\limits^{n}_{j= m+1} \frac{(\beta_j - \gamma_j)}{(\beta_j - \alpha_j)} \frac{u_{{\infty}} ^{\beta - \gamma}}{u_{j_{\infty}}}  \|u_j\|^2_2 ) 
        \\& + (\sum\limits^{m}_{i=1} d_i \frac{(\alpha_i - \gamma_i)}{(\alpha - \beta_i)}  \frac{u_{{\infty}} ^{\alpha - \gamma}}{u_{i_{\infty}}} \|\nabla{u_i}\|^2_2 
        + \sum\limits^{n}_{j= m+1} d_j \frac{(\beta_j - \gamma_j)}{(\beta_j - \alpha_j)} \frac{u_{{\infty}} ^{\beta - \gamma}}{u_{j_{\infty}}}  \|\nabla{u_j}\|^2_2 )
        \\& = \int_{\Omega}
         \prod(u_{i} + u_{i_{\infty}} )^{\gamma_i}
     [  \sum\limits^{m}_{i=1} (\alpha_i - \gamma_i) u_i \frac{u_{{\infty}} ^{\alpha - \gamma}}{u_{i_{\infty}}} 
 -  \sum\limits^{n}_{j= m+1} (\beta_j - \gamma_j) u_j \frac{u_{{\infty}} ^{\beta - \gamma}}{u_{j_{\infty}}} 
    \\& + N(u,u_{\infty})  ]
    ( - \sum\limits^{m}_{i=1} (\alpha_i - \gamma_i) u_i \frac{u_{{\infty}} ^{\alpha - \gamma}}{u_{i_{\infty}}} 
+  \sum\limits^{n}_{j= m+1} (\beta_j - \gamma_j) u_j \frac{u_{{\infty}} ^{\beta - \gamma}}{u_{j_{\infty}}}) dx
\end{split}
\end{equation}

Now we do the estimate on the right hand side of \eqref{first part energy estimate for L_2 in general}.

\begin{lemma}    \label{first right side computation}
If $\forall t \geq 0$, 
$\sum\limits^{n}_{i=1}  \|  u_i(x,t)\|_{\infty} \leq \theta
$, we have
\begin{equation}
\begin{split}   
       & \int_{\Omega}
         \prod(u_{i} + u_{i_{\infty}} )^{\gamma_i}
     [  \sum\limits^{m}_{i=1} (\alpha_i - \gamma_i) u_i \frac{u_{{\infty}} ^{\alpha - \gamma}}{u_{i_{\infty}}} 
 -  \sum\limits^{n}_{j= m+1} (\beta_j - \gamma_j) u_j \frac{u_{{\infty}} ^{\beta - \gamma}}{u_{j_{\infty}}} 
 + N(u,u_{\infty})  ]
  ( - \sum\limits^{m}_{i=1} (\alpha_i - \gamma_i)  
\frac{u_{{\infty}} ^{\alpha - \gamma}}{u_{i_{\infty}}} 
\\& u_i +  \sum\limits^{n}_{j= m+1} (\beta_j - \gamma_j) u_j \frac{u_{{\infty}} ^{\beta - \gamma}}{u_{j_{\infty}}}) dx
\lesssim  \sum\limits^{m}_{i=1} d_i \frac{(\alpha_i - \gamma_i)}{(\alpha - \beta_i)}  \frac{u_{{\infty}} ^{\alpha - \gamma}}{u_{i_{\infty}}} \|\nabla{u_i}\|^2_2 
        + \sum\limits^{n}_{j= m+1} d_j \frac{(\beta_j - \gamma_j)}{(\beta_j - \alpha_j)} \frac{u_{{\infty}} ^{\beta - \gamma}}{u_{j_{\infty}}}  \|\nabla{u_j}\|^2_2         
\end{split}
\end{equation}

\end{lemma}

\begin{proof}

Now we consider the sign situation for $\{ u_i \}, {i=1,...,n}$ in following two cases. 

\begin{enumerate}

\item The first case is when the sign for $\{u_i\}_{i \in L}$ is different from $\{u_j\}_{j \in R}$,
\begin{enumerate}
\item{} $\forall i \in L$ $ u_i \leq 0$, $\forall j \in R$ $ u_j \geq 0$.
\item{} $\forall i \in L$ $ u_i \geq 0$ , $\forall j \in R$ $ u_j \leq 0$.
\end{enumerate}

\item The rest situations belong to the second case
and we divide this case into three following parts,
\begin{enumerate}
\item{} $\{u_j\}_{j \in R}$ has positive and negative members.
\item{} $\forall j \in R$ $u_j \leq 0$, $\exists  i \in L$ such that $u_i \leq 0$.
\item{} $\forall j \in R$ $u_j \geq 0$, $\exists  i \in L$ such that $u_i \geq 0$.
\end{enumerate}

\end{enumerate}

We first deal with $2(a)$ when $ \{ u_r \}_{r \in R}$ has positive and negative members. For each $l \in L$ with $u_l \leq 0$, we further assume that $u_N \leq 0$ for $N \in \{m+1,...,o\}$ and $u_P \geq 0$ for $P \in \{o+1,...,n\}$. Recall \eqref{conservation law on steady state}, we have the following conservation laws, $\forall l \in L$, $\forall k \in R$
\begin{equation}
    \begin{split}  \label{first part general conservation law}
        \frac{1}{\alpha_l - \beta_l} \int_{\Omega} u_l(t,x) \ dx + \frac{1}{\beta_k - \alpha_k} \int_{\Omega} u_k(t,x) \ dx = 0
    \end{split}
\end{equation}
Here we define $\theta_{l,k} = \frac{1}{\alpha_l - \beta_l} u_l + \frac{1}{\beta_k - \alpha_k} u_k$. From \eqref{first part general conservation law}, we get $\int_{\Omega} \theta_{l,k}(t,x) \ dx = 0$.

For $N \in \{m+1,...,o\}$, since $u_l, u_N \leq 0$, we have
$$
|u_l| = |(\alpha_l - \beta_l) \theta_{l,N} - \frac{\alpha_l - \beta_l}{\beta_N - \alpha_N} u_N| \leq (\alpha_l - \beta_l) |\theta_{l,N}|,
$$
$$ 
|u_N| = |(\beta_N - \alpha_N) \theta_{l,N} - \frac{\beta_N - \alpha_N}{\alpha_l - \beta_l} u_l| \leq (\beta_N - \alpha_N) |\theta_{l,N}|
$$

For $P \in \{o+1,...,n\}$, since $u_P \geq 0, u_o \leq 0$, we have 
$$
\theta_{l,P} - \theta_{l,o} = \frac{1}{\beta_P - \alpha_P} u_P - \frac{1}{\beta_o - \alpha_o} u_o \geq 0
$$
this implies
$$
0 \leq u_P = (\beta_P - \alpha_P)
( \theta_{l,P} - \theta_{l,o} ) + \frac{\beta_P - \alpha_P}{\beta_o - \alpha_o} u_o \leq (\beta_P - \alpha_P)
( \theta_{l,P} - \theta_{l,o} )
$$
Combining the above two parts, we have for each $r \in R$, $l \in L$ with $u_l \leq 0$,
\begin{equation}  \label{sign estimate 2a 1}
u_r \leq \left\{
\begin{aligned}
(\beta_r - \alpha_r) |\theta_{l,r}| &, & r \in \{m+1,...,o\} \\
(\beta_r - \alpha_r)
( \theta_{l,r} - \theta_{l,o} ) &, & r \in \{o+1,...,n\}
\end{aligned}
\right.
\end{equation}
$$
|u_l| \leq (\alpha_l - \beta_l) |\theta_{l, k_{l}}|,
$$
where $k_{l} \in R$ and $u_{k_{l}}, u_l$ have the same sign.

For each $l \in L$ with $u_l \geq 0$ , recall that $ \{ u_r \}_{r \in R}$ has positive and negative members and $u_N \leq 0$ for $N \in \{m+1,...,o\}$ and $u_P \geq 0$ for $P \in \{o+1,...,n\}$ and $\theta_{l,k} = \frac{1}{\alpha_l - \beta_l} u_l + \frac{1}{\beta_k - \alpha_k} u_k$ with $\int_{\Omega} \theta_{l,k}(t,x) \ dx = 0$, $\forall k \in R$.

For $P \in \{o+1,...,n\}$, since $u_l, u_P \geq 0$, we have
$$
0 \leq u_l = (\alpha_l - \beta_l) \theta_{l,P} - \frac{\alpha_l - \beta_l}{\beta_P - \alpha_P} u_P \leq (\alpha_l - \beta_l) \theta_{l,P},
$$
$$ 
0 \leq u_P = (\beta_P -\alpha_P ) \theta_{l,P} - \frac{\beta_P -\alpha_P }{\alpha_l - \beta_l} u_l \leq (\beta_P - \alpha_P) \theta_{l,P}
$$

For $N \in \{m+1,...,o\}$, since $u_N \leq 0, u_n \geq 0$, we have 
$$
\theta_{l,N} - \theta_{l,n} = \frac{1}{\beta_N - \alpha_N} u_N - \frac{1}{\beta_n - \alpha_n} u_n \leq 0
$$
this implies
$$
|u_N| = |(\beta_N - \alpha_N)
( \theta_{l,N} - \theta_{l,n} ) + \frac{\beta_N - \alpha_N}{\beta_n - \alpha_n} u_n | \leq 
(\beta_N - \alpha_N)
| \theta_{l,N} - \theta_{l,n} |
$$
Combining the above two parts, we have for each $r \in R$, $l \in L$ with $u_l \geq 0$,
\begin{equation} \label{sign estimate 2a 2}
u_r \leq \left\{
\begin{aligned}
(\beta_r - \alpha_r)
| \theta_{l,r} - \theta_{l,o} | &, & r \in \{m+1,...,o\} \\
(\beta_r - \alpha_r) \theta_{l,r} &, & r \in \{o+1,...,n\}
\end{aligned}
\right.
\end{equation}
$$
| u_l | \leq   (\alpha_l - \beta_l) |\theta_{l,k_{l}}| 
$$
where $k_{l} \in R$ and $u_{k_{l}}, u_l$ have the same sign.

In $2(b)$ when $\forall j \in R$ $u_j \leq 0$, $\exists  i \in L$ such that $u_i \leq 0$. Then we can assume that $u_N \leq 0$ for $N \in \{1,...,q\}$ and $u_P \geq 0$ for $P \in \{q+1,...,m\}$. Again we define $\theta_{l,k} = \frac{1}{\alpha_l - \beta_l} u_l + \frac{1}{\beta_k - \alpha_k} u_k$, $\forall l \in L$, $\forall k \in R$ with $\int_{\Omega} \theta_{l,k}(t,x) \ dx = 0$ and we do the similar estimate as \eqref{sign estimate 2a 1}.

For $N \in \{1,...,q\}$, since $ u_N \leq 0$, $\forall j \in R$ $u_j \leq 0$, we have
$$
|u_N| = |(\alpha_N - \beta_N) \theta_{N,j} - \frac{\alpha_N - \beta_N}{\beta_j - \alpha_j} u_j| \leq (\alpha_N - \beta_N) |\theta_{N,j}|,
$$
$$ 
|u_j| = |(\beta_j - \alpha_j) \theta_{N,j} - \frac{\beta_j - \alpha_j}{\alpha_N - \beta_N} u_j| \leq (\beta_j - \alpha_j) |\theta_{N,j}|
$$

For $P \in \{q+1,...,m\}$, since $u_P \geq 0, u_q \leq 0$, we have 
$$
\theta_{P,j} - \theta_{q,j} = \frac{1}{\alpha_P - \beta_P} u_P - \frac{1}{\alpha_q - \beta_q} u_q \geq 0
$$
this implies
$$
0 \leq u_P = (\alpha_P - \beta_P)
( \theta_{P,j} - \theta_{q,j}) + \frac{\alpha_P - \beta_P}{\alpha_q - \beta_q} u_q \leq (\alpha_P - \beta_P)
( \theta_{P,j} - \theta_{q,j})
$$
Thus we have for each $j \in R$, $l \in L$,
\begin{equation}
u_l \leq \left\{
\begin{aligned}
(\alpha_l - \beta_l) |\theta_{l,j}| &, & l \in \{1,...,q\} \\
(\alpha_l - \beta_l)
( \theta_{l,j} - \theta_{q,j}) &, & l \in \{q+1,...,m\}
\end{aligned}
\right.
\end{equation}
$$
|u_j| \leq (\beta_j - \alpha_j) |\theta_{k_j,j}|,
$$
where $k_{j} \in L$ and $u_{k_{j}}, u_j$ have the same sign.

In $2(c)$ when $\forall j \in R$ $u_j \geq 0$, $\exists  i \in L$ such that $u_i \geq 0$. Then we can assume that $u_N \geq 0$ for $N \in \{1,...,q\}$ and $u_P \leq 0$ for $P \in \{q+1,...,m\}$. Again we define $\theta_{l,k} = \frac{1}{\alpha_l - \beta_l} u_l + \frac{1}{\beta_k - \alpha_k} u_k$, $\forall l \in L$, $\forall k \in R$ with $\int_{\Omega} \theta_{l,k}(t,x) \ dx = 0$ and we do the similar estimate as \eqref{sign estimate 2a 2}.

For $N \in \{1,...,q\}$, since $ u_N \geq 0$, $\forall j \in R$ $u_j \geq 0$, we have
$$
|u_N| = |(\alpha_N - \beta_N) \theta_{N,j} - \frac{\alpha_N - \beta_N}{\beta_j - \alpha_j} u_j| \leq (\alpha_N - \beta_N) |\theta_{N,j}|,
$$
$$ 
|u_j| = |(\beta_j - \alpha_j) \theta_{N,j} - \frac{\beta_j - \alpha_j}{\alpha_N - \beta_N} u_j| \leq (\beta_j - \alpha_j) |\theta_{N,j}|
$$

For $P \in \{q+1,...,m\}$, since $u_P \leq 0, u_q \geq 0$, we have 
$$
\theta_{P,j} - \theta_{q,j} = \frac{1}{\alpha_P - \beta_P} u_P - \frac{1}{\alpha_q - \beta_q} u_q \leq 0
$$
this implies
$$
|u_P| = | (\alpha_P - \beta_P)
( \theta_{P,j} - \theta_{q,j}) + \frac{\alpha_P - \beta_P}{\alpha_q - \beta_q} u_q | \leq  (\alpha_P - \beta_P)
|  \theta_{P,j} - \theta_{q,j} |
$$
Thus we have for each $j \in R$, $l \in L$,
\begin{equation}
u_l \leq \left\{
\begin{aligned}
(\alpha_l - \beta_l) |\theta_{l,j}| &, & l \in \{1,...,q\} \\
(\alpha_l - \beta_l)
| \theta_{l,j} - \theta_{q,j} | &, & l \in \{q+1,...,m\}
\end{aligned}
\right.
\end{equation}
$$
|u_j| \leq (\beta_j - \alpha_j) |\theta_{k_j,j}|,
$$
where $k_{j} \in L$ and $u_{k_{j}}, u_j$ have the same sign.

Recall the right hand side of \eqref{first part energy estimate for L_2 in general} and the following inequality,
\begin{equation*}
\begin{split}
& [  \sum\limits^{m}_{i=1} (\alpha_i - \gamma_i) u_i \frac{u_{{\infty}} ^{\alpha - \gamma}}{u_{i_{\infty}}} 
 -  \sum\limits^{n}_{j= m+1} (\beta_j - \gamma_j) u_j \frac{u_{{\infty}} ^{\beta - \gamma}}{u_{j_{\infty}}} 
 + N(u,u_{\infty})  ]
    ( - \sum\limits^{m}_{i=1} (\alpha_i - \gamma_i) u_i \frac{u_{{\infty}} ^{\alpha - \gamma}}{u_{i_{\infty}}}
\\& +  \sum\limits^{n}_{j= m+1} (\beta_j - \gamma_j) u_j \frac{u_{{\infty}} ^{\beta - \gamma}}{u_{j_{\infty}}})
\leq 
  \frac{1}{4} 
 (N (u,u_{\infty}))^2  
\end{split}
\end{equation*}

which implies that
\begin{equation}
\begin{split}   \label{second case nonlinear estiamte}
& \int_{\Omega}
 \prod(u_{i} + u_{i_{\infty}} )^{\gamma_i}
[  \sum\limits^{m}_{i=1} (\alpha_i - \gamma_i) u_i \frac{u_{{\infty}} ^{\alpha - \gamma}}{u_{i_{\infty}}} 
 -  \sum\limits^{n}_{j= m+1} (\beta_j - \gamma_j) u_j \frac{u_{{\infty}} ^{\beta - \gamma}}{u_{j_{\infty}}} 
    \\& + N(u,u_{\infty})  ]
    ( - \sum\limits^{m}_{i=1} (\alpha_i - \gamma_i) u_i \frac{u_{{\infty}} ^{\alpha - \gamma}}{u_{i_{\infty}}}
+  \sum\limits^{n}_{j= m+1} (\beta_j - \gamma_j) u_j \frac{u_{{\infty}} ^{\beta - \gamma}}{u_{j_{\infty}}}) dx
\\& \leq 
  \frac{1}{4} \int_{\Omega}
 \prod(u_{i} + u_{i_{\infty}} )^{\gamma_i}
 (N (u,u_{\infty}))^2  dx
\end{split}
\end{equation}
\\
Since $N(u,u_{\infty})$ is the non-linear part for $\prod(u_{i} + u_{i_{\infty}} )^{\alpha_i - \gamma_i} - \prod(u_{i} + u_{i_{\infty}} )^{\beta_i - \gamma_i}$, each component contains as least two of $\{ u_i \}, {i=1,...,n}$. So every non-linear component should be in the form of $f (u,u_{\infty}) u_i u_j $ where $f (u,u_{\infty})$ is the polynomial for $(u,u_{\infty})$ and we have the following estimate,
\begin{equation}
    \begin{split} \label{second case nonlinear component estimate}
         (f (u,u_{\infty}) u_i u_j)^2 & \leq \| f \cdot u_i \|^{2}_{\infty} \cdot u^2_j
         \\& \lesssim \| u_i \|^{2}_{\infty} \cdot
\sum\limits_{l \in L, r \in R} \theta^2_{l, r} 
    \end{split}
\end{equation}
From \eqref{second case nonlinear estiamte}, \eqref{second case nonlinear component estimate} and using Poincare inequality motivated from  $\int_{\Omega} \theta_{l,k}(t,x) \ dx = 0$, we get
\begin{equation}
\begin{split}   
&  
  \frac{1}{4} \int_{\Omega}
 \prod(u_{i} + u_{i_{\infty}} )^{\gamma_i}
 (f (u,u_{\infty}) u_i u_j)^2  dx
\\& \lesssim \prod(\|u_{i}\|_{\infty} + u_{i_{\infty}} )^{\gamma_i} \| u_i \|^{2}_{\infty} \sum\limits_{l \in L, r \in R} \int_{\Omega} \theta^2_{l,r} dx
\\& \lesssim \| u_1 \|^{2}_{\infty} \sum\limits_{l \in L, r \in R} \int_{\Omega} \nabla^2{\theta_{l,r}} dx \lesssim \| u_i \|^{2}_{\infty} \sum\limits_{l \in L, r \in R} ( \|\nabla{u_l}\|^2_2 + \|\nabla{u_{r}}\|^2_2 )
\end{split}
\end{equation}
We can do the similar estimate on all non-linear components of $N(u, u_{\infty})$ as above.
Therefore as long as $\sum\limits_{i=1}^{n} \| u_i \|_{\infty} \leq \theta$ are sufficiently small such that
$\forall i \in L$, $\forall j \in R$, 
$$
\frac{1}{4} 
 \prod(\|u_{i}\|_{\infty} + u_{i_{\infty}} )^{\gamma_i}
 \|f(\theta, u_{i_{\infty}})\|^2_{\infty} \theta^2 \leq \min \{d_i \frac{(\alpha_i - \gamma_i)}{(\alpha - \beta_i)}  \frac{u_{{\infty}} ^{\alpha - \gamma}}{u_{i_{\infty}}}, d_j \frac{(\beta_j - \gamma_j)}{(\beta_j - \alpha_j)} \frac{u_{{\infty}} ^{\beta - \gamma}}{u_{j_{\infty}}} \},
$$ we can get
\begin{equation}
\begin{split}   \label{second case energy estimate for L_2 in general}
        &  \sum\limits^{m}_{i=1} d_i \frac{(\alpha_i - \gamma_i)}{(\alpha - \beta_i)}  \frac{u_{{\infty}} ^{\alpha - \gamma}}{u_{i_{\infty}}} \|\nabla{u_i}\|^2_2 
        + \sum\limits^{n}_{j= m+1} d_j \frac{(\beta_j - \gamma_j)}{(\beta_j - \alpha_j)} \frac{u_{{\infty}} ^{\beta - \gamma}}{u_{j_{\infty}}}  \|\nabla{u_j}\|^2_2 
        \\& \geq \int_{\Omega}
         \prod(u_{i} + u_{i_{\infty}} )^{\gamma_i}
     [  \sum\limits^{m}_{i=1} (\alpha_i - \gamma_i) u_i \frac{u_{{\infty}} ^{\alpha - \gamma}}{u_{i_{\infty}}} 
 -  \sum\limits^{n}_{j= m+1} (\beta_j - \gamma_j) u_j \frac{u_{{\infty}} ^{\beta - \gamma}}{u_{j_{\infty}}} 
    \\& + N(u,u_{\infty})  ]
    ( - \sum\limits^{m}_{i=1} (\alpha_i - \gamma_i) u_i \frac{u_{{\infty}} ^{\alpha - \gamma}}{u_{i_{\infty}}} 
+  \sum\limits^{n}_{j= m+1} (\beta_j - \gamma_j) u_j \frac{u_{{\infty}} ^{\beta - \gamma}}{u_{j_{\infty}}}) dx
\end{split}
\end{equation}

In the first case, we first consider $1(a)$ when $\forall i \in L$ $ u_i \leq 0$, $\forall j \in R$ $ u_j \geq 0$. 
This implies
$$
\sum\limits^{m}_{i=1} (\alpha_i - \gamma_i) u_i \frac{u_{{\infty}} ^{\alpha - \gamma}}{u_{i_{\infty}}} 
 -  \sum\limits^{n}_{j= m+1} (\beta_j - \gamma_j) u_j \frac{u_{{\infty}} ^{\beta - \gamma}}{u_{j_{\infty}}} \leq 0
$$
Recall $N(u,u_{\infty})$ is the non-linear part and each component contains as least two of $\{ u_i \}_{i=1,...,n}$, as long as $\| u_i \|_{\infty}$ are sufficiently small, we can get
$$
\sum\limits^{m}_{i=1} (\alpha_i - \gamma_i) u_i \frac{u_{{\infty}} ^{\alpha - \gamma}}{u_{i_{\infty}}} 
 -  \sum\limits^{n}_{j= m+1} (\beta_j - \gamma_j) u_j \frac{u_{{\infty}} ^{\beta - \gamma}}{u_{j_{\infty}}} 
 + N(u,u_{\infty}) \leq 0
$$
Also recall the right hand side of \eqref{first part energy estimate for L_2 in general}, we get
\begin{equation}
\begin{split}   \label{first case energy estimate for L_2 in general}
& \int_{\Omega}
 \prod(u_{i} + u_{i_{\infty}} )^{\gamma_i}
[  \sum\limits^{m}_{i=1} (\alpha_i - \gamma_i) u_i \frac{u_{{\infty}} ^{\alpha - \gamma}}{u_{i_{\infty}}} 
 -  \sum\limits^{n}_{j= m+1} (\beta_j - \gamma_j) u_j \frac{u_{{\infty}} ^{\beta - \gamma}}{u_{j_{\infty}}} 
    \\& + N(u,u_{\infty})  ]
    ( - \sum\limits^{m}_{i=1} (\alpha_i - \gamma_i) u_i \frac{u_{{\infty}} ^{\alpha - \gamma}}{u_{i_{\infty}}}
+  \sum\limits^{n}_{j= m+1} (\beta_j - \gamma_j) u_j \frac{u_{{\infty}} ^{\beta - \gamma}}{u_{j_{\infty}}}) dx
 \leq 0
\end{split}
\end{equation}
the above estimate also works for $1(b)$ when $\forall i \in L, u_i \geq 0$ and $\forall j \in R, u_j \leq 0$. 

\end{proof}

Combining \eqref{second case energy estimate for L_2 in general} \eqref{first case energy estimate for L_2 in general} and the equation \eqref{first part energy estimate for L_2 in general}, we get the first part of energy estimate

\begin{lemma}  \label{first right hand side estimate}
If $\forall t \geq 0$, 
$\sum\limits^{n}_{i=1}  \|  u_i(x,t)\|_{\infty} \leq \theta
$, then we have

\begin{equation} 
    \frac{d}{dt}( \sum\limits^{m}_{i=1} (\frac{(\alpha_i - \gamma_i)}{(\alpha - \beta_i)}  \frac{u_{{\infty}} ^{\alpha - \gamma}}{u_{i_{\infty}}} \|u_i\|^2_2
 +  \sum\limits^{n}_{j= m+1} \frac{(\beta_j - \gamma_j)}{(\beta_j - \alpha_j)} \frac{u_{{\infty}} ^{\beta - \gamma}}{u_{j_{\infty}}}  \|u_j\|^2_2 ) 
   \leq   0
\end{equation}
, this implies $\sum\limits^{n}_{i=1}  \|  u_i(x,t)\|_2$ decay w.r.t time.

\end{lemma}

In order to use the elliptic estimate in Theorem \ref{nirenberg}, we need to do the energy estimate on $\| \partial_t u_i\|_2$. 
By taking time partial derivative on \eqref{equation in general}, 
multiplying $ \frac{(\alpha_i - \gamma_i)}{(\alpha_i - \beta_i)}  \frac{u_{{\infty}} ^{\alpha - \gamma}}{u_{i_{\infty}}} \partial_t u_i$, $\frac{(\beta_j - \gamma_j)}{(\beta_j - \alpha_j)} \frac{u_{{\infty}} ^{\beta - \gamma}}{u_{j_{\infty}}} \partial_t u_j$ respectively and  integrating over $\Omega$, we get the following

\begin{equation}
\begin{split}   \label{second part energy estimate for L_2 in general}
        & \frac{1}{2}\frac{d}{dt}( \sum\limits^{m}_{i=1} (\frac{(\alpha_i - \gamma_i)}{(\alpha - \beta_i)}  \frac{u_{{\infty}} ^{\alpha - \gamma}}{u_{i_{\infty}}} \| \partial_t u_i\|^2_2
 +  \sum\limits^{n}_{j= m+1} \frac{(\beta_j - \gamma_j)}{(\beta_j - \alpha_j)} \frac{u_{{\infty}} ^{\beta - \gamma}}{u_{j_{\infty}}}  \| \partial_t u_j\|^2_2 ) 
        \\& + (\sum\limits^{m}_{i=1} d_i \frac{(\alpha_i - \gamma_i)}{(\alpha - \beta_i)}  \frac{u_{{\infty}} ^{\alpha - \gamma}}{u_{i_{\infty}}} \|\nabla{ \partial_t u_i}\|^2_2 
        + \sum\limits^{n}_{j= m+1} d_j \frac{(\beta_j - \gamma_j)}{(\beta_j - \alpha_j)} \frac{u_{{\infty}} ^{\beta - \gamma}}{u_{j_{\infty}}}  \|\nabla{ \partial_t u_j}\|^2_2 )
\\& = \uppercase\expandafter{\romannumeral1} + \uppercase\expandafter{\romannumeral2} 
\end{split}
\end{equation}
where $ \uppercase\expandafter{\romannumeral1} = \int_{\Omega}
         \prod(u_{i} + u_{i_{\infty}} )^{\gamma_i}
     [  \sum\limits^{m}_{i=1} (\alpha_i - \gamma_i)   \frac{u_{{\infty}} ^{\alpha - \gamma}}{u_{i_{\infty}}} \partial_t u_i
 -  \sum\limits^{n}_{j= m+1} (\beta_j - \gamma_j)   \frac{u_{{\infty}} ^{\beta - \gamma}}{u_{j_{\infty}}} \partial_t u_j 
     +  \partial_t N(u,u_{\infty})  ]
   ( - \sum\limits^{m}_{i=1} (\alpha_i - \gamma_i)   \frac{u_{{\infty}} ^{\alpha - \gamma}}{u_{i_{\infty}}} \partial_t u_i
+  \sum\limits^{n}_{j= m+1} (\beta_j - \gamma_j)   \frac{u_{{\infty}} ^{\beta - \gamma}}{u_{j_{\infty}}} \partial_t u_j ) dx$  
and  $ \uppercase\expandafter{\romannumeral2}  =  \int_{\Omega}
      \partial_t \{ \prod(u_{i} + u_{i_{\infty}} )^{\gamma_i} \}
     [  \sum\limits^{m}_{i=1} (\alpha_i - \gamma_i) u_i \frac{u_{{\infty}} ^{\alpha - \gamma}}{u_{i_{\infty}}} 
 -  \sum\limits^{n}_{j= m+1} (\beta_j - \gamma_j) u_j \frac{u_{{\infty}} ^{\beta - \gamma}}{u_{j_{\infty}}} 
 + N(u,u_{\infty})]   
 ( - \sum\limits^{m}_{i=1} (\alpha_i - \gamma_i)   \frac{u_{{\infty}} ^{\alpha - \gamma}}{u_{i_{\infty}}} \partial_t u_i
+  \sum\limits^{n}_{j= m+1} (\beta_j - \gamma_j)   \frac{u_{{\infty}} ^{\beta - \gamma}}{u_{j_{\infty}}} \partial_t u_j ) dx$.

The idea for the proof in the following Lemma is similar to the estimate in Lemma \ref{first right side computation}.

\begin{lemma}  \label{second right side computation}
If $\forall t \geq 0$, 
$\sum\limits^{n}_{i=1}  \|  u_i(x,t)\|_{\infty} \leq \theta
$, we have

\begin{equation}
\begin{split}   
        & \uppercase\expandafter{\romannumeral1} + \uppercase\expandafter{\romannumeral2} \leq
 \sum\limits^{m}_{i=1} d_i \frac{(\alpha_i - \gamma_i)}{(\alpha - \beta_i)}  \frac{u_{{\infty}} ^{\alpha - \gamma}}{u_{i_{\infty}}} \|\nabla{ \partial_t u_i}\|^2_2 
        + \sum\limits^{n}_{j= m+1} d_j \frac{(\beta_j - \gamma_j)}{(\beta_j - \alpha_j)} \frac{u_{{\infty}} ^{\beta - \gamma}}{u_{j_{\infty}}}  \|\nabla{ \partial_t u_j}\|^2_2 
\end{split}
\end{equation}

\end{lemma}

\begin{proof}

Again we consider the sign situation for $\{ \partial_t u_i \}, {i=1,...,n}$  in two cases. 

\begin{enumerate}

\item The first case is when the sign for $\{\partial_t u_i\}_{i \in L}$ is different from $\{\partial_t u_j\}_{j \in R}$,
\begin{enumerate}
\item{} $\forall i \in L$ $ \partial_t u_i \leq 0$, $\forall j \in R$ $ \partial_t u_j \geq 0$.
\item{} $\forall i \in L$ $ \partial_t u_i \geq 0$ , $\forall j \in R$ $ \partial_t u_j \leq 0$.
\end{enumerate}

\item The rest situations belong to the second case
and we divide this case into three following parts,
\begin{enumerate}
\item{} $\{u_j\}_{j \in R}$ has positive and negative members.
\item{} $\forall j \in R$, $\partial_t u_j \leq 0$, $\exists  i \in L$ such that $\partial_t u_i \leq 0$.
\item{} $\forall j \in R$, $\partial_t u_j \geq 0$, $\exists  i \in L$ such that $\partial_t u_i \geq 0$.
\end{enumerate}

\end{enumerate}

We first deal with the second case, for each $l \in L$ with $ \partial_t u_l \leq 0$. The assumption implies either $ \{  \partial_t u_r \}_{r \in R}$ have different signs or $\forall j \in R$, $\partial_t u_j \leq 0$. W.l.o.g. we assume $ \partial_t u_N \leq 0$ for $N \in \{m+1,...,o\}$ and $ \partial_t u_P \geq 0$ for $P \in \{o+1,...,n\}$. Recall \eqref{first part general conservation law}, we have the similar conservation laws for $\partial_t u_i$, $\forall k \in R$
\begin{equation}
    \begin{split}  \label{second part general conservation law}
        \frac{1}{\alpha_l - \beta_l} \int_{\Omega}  \partial_t u_l \ dx + \frac{1}{\beta_k - \alpha_k} \int_{\Omega}  \partial_t u_k \ dx = 0
    \end{split}
\end{equation}
\\
Here we define $\theta^t_{l,k} = \frac{1}{\alpha_l - \beta_l}  \partial_t u_l + \frac{1}{\beta_k - \alpha_k}  \partial_t u_k$ and $\int_{\Omega} \theta^t_{l,k} \ dx = 0$.

For $N \in \{m+1,...,o\}$, we have
$$
| \partial_t u_l| \leq (\alpha_l - \beta_l) |\theta^t_{l,N}|, 
| \partial_t u_N| \leq (\beta_N - \alpha_N) |\theta^t_{l,N}|
$$

For $P \in \{o+1,...,n\}$, we have 
$$
\theta^t_{l,P} - \theta^t_{l,o} = \frac{1}{\beta_P - \alpha_P}  \partial_t u_P - \frac{1}{\beta_o - \alpha_o}  \partial_t u_o
$$
this implies
$$
0 \leq  \partial_t u_P \leq (\beta_P - \alpha_P)
( \theta_{l,P} - \theta_{l,o} )
$$
Combining the above two parts, we have for each $r \in R$,
\begin{equation}
 \partial_t u_r \leq \left\{
\begin{aligned}
(\beta_r - \alpha_r) |\theta^t_{l,r}| &, & r \in \{m+1,...,o\} \\
(\beta_r - \alpha_r)
( \theta^t_{l,r} - \theta^t_{l,o} ) &, & r \in \{o+1,...,n\}
\end{aligned}
\right.
\end{equation}

Then for each $l \in L$ with $ \partial_t u_l \geq 0$, the assumption again implies either $ \{  \partial_t u_r \}_{r \in R}$ have different signs or $\forall j \in R$, $\partial_t u_j \leq 0$. We can get the similar estimate, for each $l \in L$,
$$
|  \partial_t u_l | \leq   (\alpha_l - \beta_l) |\theta^t_{l,k_{l}}| 
$$
where $k_{l} \in R$ and $ \partial_t u_{k_{l}},  \partial_t u_l$ have the same sign.

Recall \eqref{second case nonlinear component estimate}, we can do the similar estimate on the right hand side of \eqref{second part energy estimate for L_2 in general}, since 
\begin{equation*}
\begin{split}
& \uppercase\expandafter{\romannumeral1} = \int_{\Omega}
         \prod(u_{i} + u_{i_{\infty}} )^{\gamma_i}
     [  \sum\limits^{m}_{i=1} (\alpha_i - \gamma_i)   \frac{u_{{\infty}} ^{\alpha - \gamma}}{u_{i_{\infty}}} \partial_t u_i
 -  \sum\limits^{n}_{j= m+1} (\beta_j - \gamma_j)   \frac{u_{{\infty}} ^{\beta - \gamma}}{u_{j_{\infty}}} \partial_t u_j 
\\&     +  \partial_t N(u,u_{\infty})  ]
    ( - \sum\limits^{m}_{i=1} (\alpha_i - \gamma_i)   \frac{u_{{\infty}} ^{\alpha - \gamma}}{u_{i_{\infty}}} \partial_t u_i
+  \sum\limits^{n}_{j= m+1} (\beta_j - \gamma_j)   \frac{u_{{\infty}} ^{\beta - \gamma}}{u_{j_{\infty}}} \partial_t u_j ) dx
\\& \leq 
\frac{1}{4} \int_{\Omega}
         \prod(u_{i} + u_{i_{\infty}} )^{\gamma_i}
     ( \partial_t N(u,u_{\infty}) )^2 dx
\end{split}
\end{equation*} 
and  
\begin{equation*}
\begin{split}
& \uppercase\expandafter{\romannumeral2}  =  \int_{\Omega}
      \partial_t \{ \prod(u_{i} + u_{i_{\infty}} )^{\gamma_i} \}
     [  \sum\limits^{m}_{i=1} (\alpha_i - \gamma_i) u_i \frac{u_{{\infty}} ^{\alpha - \gamma}}{u_{i_{\infty}}} 
 -  \sum\limits^{n}_{j= m+1} (\beta_j - \gamma_j) u_j \frac{u_{{\infty}} ^{\beta - \gamma}}{u_{j_{\infty}}} 
\\& + N(u,u_{\infty})]   
 ( - \sum\limits^{m}_{i=1} (\alpha_i - \gamma_i)   \frac{u_{{\infty}} ^{\alpha - \gamma}}{u_{i_{\infty}}} \partial_t u_i
+  \sum\limits^{n}_{j= m+1} (\beta_j - \gamma_j)   \frac{u_{{\infty}} ^{\beta - \gamma}}{u_{j_{\infty}}} \partial_t u_j )  dx
\\& \leq \| \sum\limits^{m}_{i=1} (\alpha_i - \gamma_i) u_i \frac{u_{{\infty}} ^{\alpha - \gamma}}{u_{i_{\infty}}} 
 -  \sum\limits^{n}_{j= m+1} (\beta_j - \gamma_j) u_j \frac{u_{{\infty}} ^{\beta - \gamma}}{u_{j_{\infty}}}  + N(u,u_{\infty}) \|_{\infty} \cdot \int_{\Omega}
\\&  | \partial_t \{ \prod(u_{i} + u_{i_{\infty}})^{\gamma_i} \}
 ( - \sum\limits^{m}_{i=1} (\alpha_i - \gamma_i)   \frac{u_{{\infty}} ^{\alpha - \gamma}}{u_{i_{\infty}}} \partial_t u_i
+  \sum\limits^{n}_{j= m+1} (\beta_j - \gamma_j)   \frac{u_{{\infty}} ^{\beta - \gamma}}{u_{j_{\infty}}} \partial_t u_j ) | \ dx
\end{split}
\end{equation*} 
where $\partial_t \{ \prod(u_{i} + u_{i_{\infty}})^{\gamma_i} \} = \sum \gamma_i  \frac{u_{\infty}^{\gamma}}{u_{i_{\infty}}} \partial_t u_i + N^{\gamma} (u,\partial_t u, u_{\infty})$ and $N^{\gamma}$ is the non-linear part for $\partial_t \{ \prod(u_{i} + u_{i_{\infty}})^{\gamma_i} \}$. By using Poincare inequality motivated from $\forall l \in L$, $\forall k \in R$, $\int_{\Omega} \theta^t_{l,k}(t,x) \ dx = 0$ and the smallness of $\| u_i \|_{\infty}$, we can get
\begin{equation}    \label{second part second case energy estimate in genral}
\sum\limits^{m}_{i=1} d_i \frac{(\alpha_i - \gamma_i)}{(\alpha - \beta_i)}  \frac{u_{{\infty}} ^{\alpha - \gamma}}{u_{i_{\infty}}} \|\nabla{ \partial_t u_i}\|^2_2 
        + \sum\limits^{n}_{j= m+1} d_j \frac{(\beta_j - \gamma_j)}{(\beta_j - \alpha_j)} \frac{u_{{\infty}} ^{\beta - \gamma}}{u_{j_{\infty}}}  \|\nabla{ \partial_t u_j}\|^2_2 \geq \uppercase\expandafter{\romannumeral1} + \uppercase\expandafter{\romannumeral2}
\end{equation}

In the first case, we first consider when $\forall i \in L, \partial_t u_i \leq 0$ and $\forall j \in R, \partial_t u_j \geq 0$. 
This implies
$$
 \sum\limits^{m}_{i=1} (\alpha_i - \gamma_i)   \frac{u_{{\infty}} ^{\alpha - \gamma}}{u_{i_{\infty}}} \partial_t u_i
-  \sum\limits^{n}_{j= m+1} (\beta_j - \gamma_j)   \frac{u_{{\infty}} ^{\beta - \gamma}}{u_{j_{\infty}}} \partial_t u_j \leq 0
$$
Then we can write
\begin{equation}
    \begin{split}   \label{second part first case sign estimate in general}
&  \uppercase\expandafter{\romannumeral1} + \uppercase\expandafter{\romannumeral2} 
   = \int_{\Omega}
      \uppercase\expandafter{\romannumeral3} \cdot
    ( - \sum\limits^{m}_{i=1} (\alpha_i - \gamma_i)   \frac{u_{{\infty}} ^{\alpha - \gamma}}{u_{i_{\infty}}} \partial_t u_i
+  \sum\limits^{n}_{j= m+1} (\beta_j - \gamma_j)   \frac{u_{{\infty}} ^{\beta - \gamma}}{u_{j_{\infty}}} \partial_t u_j ) dx     
    \end{split}
\end{equation}
where $\uppercase\expandafter{\romannumeral3} =     \prod(u_{i} + u_{i_{\infty}} )^{\gamma_i}
     [  \sum\limits^{m}_{i=1} (\alpha_i - \gamma_i)   \frac{u_{{\infty}} ^{\alpha - \gamma}}{u_{i_{\infty}}} \partial_t u_i
 -  \sum\limits^{n}_{j= m+1} (\beta_j - \gamma_j)    
    \frac{u_{{\infty}} ^{\beta - \gamma}}{u_{j_{\infty}}} \partial_t u_j  +
  + \partial_t N(u,u_{\infty})  ] 
+ \partial_t \{ \prod(u_{i} + u_{i_{\infty}} )^{\gamma_i} \}
     [  \sum\limits^{m}_{i=1} (\alpha_i - \gamma_i)  \frac{u_{{\infty}} ^{\alpha - \gamma}}{u_{i_{\infty}}} u_i
 -  \sum\limits^{n}_{j= m+1} (\beta_j - \gamma_j) 
 \frac{u_{{\infty}} ^{\beta - \gamma}}{u_{j_{\infty}}}   u_j  
 + N(u,u_{\infty})] $  and because of the smallness of $\|u_i\|_{\infty}$ the value (sign) of $\uppercase\expandafter{\romannumeral3}$ is controlled by $$
u_{{\infty}} ^{\gamma}
     [  \sum\limits^{m}_{i=1} (\alpha_i - \gamma_i)   \frac{u_{{\infty}} ^{\alpha - \gamma}}{u_{i_{\infty}}} \partial_t u_i
 -  \sum\limits^{n}_{j= m+1} (\beta_j - \gamma_j)    
    \frac{u_{{\infty}} ^{\beta - \gamma}}{u_{j_{\infty}}} \partial_t u_j  ]  \leq 0
$$
Therefore \eqref{second part first case sign estimate in general} and the above inequality implies that in the first case 
\begin{equation} \label{second part first case energy estimate in genral}
    \uppercase\expandafter{\romannumeral1} + \uppercase\expandafter{\romannumeral2} \leq 0
\end{equation}

\end{proof}

Combining \eqref{second part second case energy estimate in genral} \eqref{second part first case energy estimate in genral} and the equation \eqref{second part energy estimate for L_2 in general}, we get the second part of energy estimate

\begin{lemma}  \label{second right hand side estimate}
If $\forall t \geq 0$, 
$\sum\limits^{n}_{i=1}  \|  u_i(x,t)\|_{\infty} \leq \theta
$, then we have

\begin{equation} 
\frac{d}{dt}( \sum\limits^{m}_{i=1} (\frac{(\alpha_i - \gamma_i)}{(\alpha - \beta_i)}  \frac{u_{{\infty}} ^{\alpha - \gamma}}{u_{i_{\infty}}} \| \partial_t u_i\|^2_2
 +  \sum\limits^{n}_{j= m+1} \frac{(\beta_j - \gamma_j)}{(\beta_j - \alpha_j)} \frac{u_{{\infty}} ^{\beta - \gamma}}{u_{j_{\infty}}}  \| \partial_t u_j\|^2_2 )  
 \leq 0
\end{equation}
, this implies $\sum\limits^{n}_{i=1}  \| \partial_t  u_i(x,t)\|_2$ decay w.r.t time.

\end{lemma}

Finally, we proof Theorem \ref{stable theorem in general} by Lemma \ref{first right hand side estimate} and Lemma \ref{second right hand side estimate}.

\begin{proof}

We first do the elliptic estimate for the system \eqref{system n}. It's not hard to check that the system satisfies the Supplementary Condition and the Neumann boundary condition satisfies the Complementing Boundary Condition.
By using Theorem \ref{nirenberg}, we have for $i=1,...,n$,
\begin{equation}
\begin{split}     \label{elliptic estimate in general}
& \|u_i\|_{H_2} 
\leq K (\| ( (u + u_{\infty})^{\gamma} - u_{\infty}^{\gamma} )
     [  \sum\limits^{m}_{i=1} (\alpha_i - \gamma_i) u_i \frac{u_{{\infty}} ^{\alpha - \gamma}}{u_{i_{\infty}}} 
 -  \sum\limits^{n}_{j= m+1} (\beta_j - \gamma_j) u_j \frac{u_{{\infty}} ^{\beta - \gamma}}{u_{j_{\infty}}} 
\\& + N (u,u_{\infty}) + u_{\infty}^{\gamma} N (u,u_{\infty}) ]  \|_2
+ \sum\limits_{i=1}^{n}\|\partial_t u_i \|_2
+ \sum\limits_{i=1}^{n}\| u_i \|_2))
\end{split}
\end{equation}
where $K$ is a constant depends on origin equation and bounded domain.
By using Sobolev Embedding Inequality, we can have
$$
\|v_i\|_{L^{\infty}} \lesssim  \sum\limits_{i=1}^{n}\|u_i\|_2 + \sum\limits_{i=1}^{n}\|\partial_t u_i\|_2
$$
the above holds because $\| u_i \|_{\infty}$ is sufficiently small which guarantees $(u + u_{\infty})^{\gamma} - u_{\infty}^{\gamma}$, $N (u,u_{\infty}) \ll 1$.

The continuity argument implies $L^{\infty}$ will be always small to follow Lemma \ref{first right hand side estimate} and Lemma \ref{second right hand side estimate}. 
As long as the initial $L^2$ on $\{ \partial_t u_i \}$ and $L^{\infty}$ on $\{ u_i \}$ are sufficiently small, $L^{\infty}$ can keep being small along with the time $t$ while $L^2$ is non-increasing from the estimate which implies the existence of weak solution around the positive equilibrium. 

The \textit{Remark 3.1} in \cite{PSU17} shows that for a reversible reaction with nonnegative initial data in $L^1\cap L\log L$ if the solution is globally (in time) essentially bounded, the solution converges exponentially to the complex-balanced equilibrium in $L^1$ norm. 
By using the interpolation with $L^1$ and boundness of $L^{\infty}$, we can get the exponential convergence in $L^p (1 < p < \infty)$ sense.

Now we return to the origin equation on $\{u_i\}_{i=1,...,n}$,

$$
 \partial_t u_{i}- d_i \Delta u_{i} = (\beta_i - \alpha_i)(\Tilde{u}^{\alpha} - \Tilde{u}^{\beta})
$$

Because of the Poincare inequality, we have
$$
\|\partial_t  u - \bar{\partial_t  u} \|_{L^2} \lesssim \|\nabla \partial_t u \|_{L^2}
$$

This implies
$$
\|\partial_t u_i\|_{L^2} \lesssim \|\nabla \partial_t u_i\|_{L^2} + |\int_{\Omega} \partial_t u_i dx|
$$

From the equation, since we know $
\sum\limits^{n}_{i=1} (\| \partial_t u_i(x,0)\|_2 + \|  u_i(x,0)\|_{\infty}) \leq \theta \ll 1
$,
\begin{equation}
\begin{split}
    & |\int_{\Omega} \partial_t u_i dx| = |d_i \int_{\Omega} \Delta u_i dx + \int_{\Omega} (\beta_i - \alpha_i)(\Tilde{u}^{\alpha} - \Tilde{u}^{\beta}) dx|
    \\& \lesssim u_{\infty}^{\gamma} 
     \|  \sum\limits^{m}_{i=1} (\alpha_i - \gamma_i) u_i \frac{u_{{\infty}} ^{\alpha - \gamma}}{u_{i_{\infty}}} 
 -  \sum\limits^{n}_{j= m+1} (\beta_j - \gamma_j) u_j \frac{u_{{\infty}} ^{\beta - \gamma}}{u_{j_{\infty}}} \|_{L^1} \lesssim e^{-lt}
\end{split}
\end{equation}
where exponential decaying rate $l$ is determined from the interpolation. 
Recall the estimate in Lemma \ref{second right hand side estimate} where we get 
\begin{equation}
\begin{split}   \label{second part energy estimate for L_2 at last}
        & \frac{1}{2}\frac{d}{dt}( \sum\limits^{m}_{i=1} (\frac{(\alpha_i - \gamma_i)}{(\alpha - \beta_i)}  \frac{u_{{\infty}} ^{\alpha - \gamma}}{u_{i_{\infty}}} \| \partial_t u_i\|^2_2
 +  \sum\limits^{n}_{j= m+1} \frac{(\beta_j - \gamma_j)}{(\beta_j - \alpha_j)} \frac{u_{{\infty}} ^{\beta - \gamma}}{u_{j_{\infty}}}  \| \partial_t u_j\|^2_2 ) 
        \\& + (\sum\limits^{m}_{i=1} d_i \frac{(\alpha_i - \gamma_i)}{(\alpha - \beta_i)}  \frac{u_{{\infty}} ^{\alpha - \gamma}}{u_{i_{\infty}}} \|\nabla{ \partial_t u_i}\|^2_2 
        + \sum\limits^{n}_{j= m+1} d_j \frac{(\beta_j - \gamma_j)}{(\beta_j - \alpha_j)} \frac{u_{{\infty}} ^{\beta - \gamma}}{u_{j_{\infty}}}  \|\nabla{ \partial_t u_j}\|^2_2 )
\\& \leq 0 
\end{split}
\end{equation}

Then we can have the following
$$
\frac{1}{2}\frac{d}{dt}(\sum\limits^{n}_{i=1} \| \partial_t u_i\|^2_2) + (\sum\limits^{n}_{i=1} \| \partial_t u_i\|^2_2) \lesssim \sum\limits^{n}_{i=1} |\int_{\Omega} \partial_t u_i dx|
\lesssim e^{-lt}
$$
The Gronwall's inequality implies that $\sum\limits^{n}_{i=1} \| \partial_t u_i\|^2_2$ decays exponentially.
Then the  elliptic estimate \eqref{elliptic estimate in general}
implies exponential convergence to positive equilibrium in $H^{2}$  sense.

\end{proof}

\section{Acknowledgment}

This work is a part of the author’s thesis. He thanks his advisor Chanwoo Kim for suggesting this problem and discussions. This research is partly supported by NSF DMS-1501031, DMS-1900923. We also thank the anonymous referees for their suggestions to improve the presentation.



\begin{thebibliography}{99}


\bibitem{Strauss.1992}
Strauss, Walter A.
{\em Partial Differential Equations: An Introduction,}
John Wiley \& Sons, Inc., New York, 1992.




\bibitem{ADN.1964}
Agmon, S. and Douglis, A. and Nirenberg, L.
{\em Estimates near the boundary for solutions of elliptic partial
differential equations satisfying general boundary conditions.{II},}
Comm. Pure Appl. Math., no. 17 (1964), pp. 35--92.



\bibitem{GJCA2019}
Craciun, G. and Jin, J. and Pantea, C. and Tudorascu, A.  
{\em Convergence to the complex balanced equilibrium for some chemical reaction-diffusion systems with boundary equilibria,} {arXiv:1812.07707}, 2019.




	
\bibitem{CMT2019}
Cupps B. P., Morgan J., Tang B. Q.
{\em Uniform boundedness for reaction-diffusion systems with mass dissipation,} {arXiv:1905.10599}, 2019.	


\bibitem{PSU17}
{ M. Pierre, T. Suzuki, H. Umakoshi},
{\em Asymptotic behavior in chemical reaction-diffusion systems with boundary equilibria,} 
J. Appl. Anal. Comp. {\bf 8}, no. 3 (2018), pp. 836--858. 

	
\bibitem{Guo-Strauss}
Guo, Y. and Strauss, W. A. 
{\em Instability of periodic BGK equilibria,}
Comm. Pure Appl. Math., 48: 861-894, 1995.

\bibitem{GHS}
Guo, Y., Hallstrom, C., Spirn, D.
{\em Dynamics near Unstable, Interfacial Fluids,}
Communications in Mathematical Physics Vol. 270, Issue 3 (2007), pp. 635--689.

\bibitem{Bao2018}
Tang, B.Q. 
{\em Close-to-equilibrium regularity for reaction–diffusion systems,}
J. Evol. Equ. (2018) 18: 845.	




\bibitem{DFPV07}{ L. Desvillettes, K. Fellner, M. Pierre, J. Vovelle}, {\em About Global Existence for Quadratic Systems of Reaction-Diffusion,} J. Adv. Nonlinear Stud. {\bf 7} (2007), pp. 491--511.

\bibitem{DFT2}{ L. Desvillettes, K. Fellner, B.Q. Tang}, 
{\em Trend to equilibrium for Reaction-Diffusion system arising from Complex Balanced Chemical Reaction Networks,} SIAM J. Math. Anal. {\bf 49}, no. 4 (2017), 2666--2709.

\bibitem{Feinberg.1972}
M. Feinberg.
{\em Complex balancing in general kinetic systems,}
Archive for Rational Mechanics and Analysis 49:3,
1972.


\bibitem{Mincheva} {M. Mincheva, Maya, D. Siegel}, {\em Stability of mass action reaction--diffusion systems, Nonlinear Analysis: Theory, Methods \& Applications,} {\bf 8} (2004), pp. 1105--1131.



\bibitem{FT18} 
K. Fellner, B.Q. Tang,
{\em Convergence to equilibrium of renormalised solutions to nonlinear chemical reaction–diffusion systems,}
Z. Angew. Math. Phys. (2018) 69.3, pp. 1--30. https://doi.org/10.1007/s00033-018-0948-3.


\bibitem{Craciun.gac} G.~Craciun. \newblock Toric Differential Inclusions and a Proof of the Global Attractor Conjecture, \newblock {arXiv:1501.02860}, 2016.


\bibitem{Horn.1972}
F. Horn, R. Jackson,
{\em General mass action kinetics,}
Archive for Rational Mechanics and Analysis 47,  1972.

\bibitem{Horn.1974}
F. Horn, 
{\em The dynamics of open reaction systems,} 
Mathematical aspects of chemical and biochemical problems and quantum chemistry, 
SIAM-AMS Proceedings Vol. VIII,  1974. 





	




\end{thebibliography}
\end{document}